\documentclass{article}
\usepackage[utf8]{inputenc}

\title{A characterization of simultaneous optimization, majorization, and (bi)submodular polyhedra}
\author{Martijn H. H. Schoot Uiterkamp \\ Tilburg University}
\date{\today}

\usepackage{graphicx}

\usepackage{booktabs} 
\usepackage{array} 
\usepackage{paralist} 
\usepackage{verbatim} 
\usepackage{subfig} 
\usepackage{amsmath}
\usepackage{amsfonts}
\usepackage{amssymb}
\usepackage{amsthm}
\usepackage{algorithm}
\usepackage{algorithmic}
\usepackage{comment}
\usepackage[disable]{todonotes}
\usepackage{multirow}
\usepackage{tabularx}

\usepackage{xcolor}

\newtheorem{theorem}{Theorem}
\newtheorem{lemma}{Lemma}
\newtheorem{corollary}{Corollary}
\newtheorem{condition}{Condition}

\theoremstyle{remark}
\newtheorem{remark}{Remark}

\usepackage[a4paper,margin=2.54cm]{geometry}

\begin{document}

\maketitle

\begin{abstract}
Motivated by resource allocation problems (RAPs) in power management applications, we investigate solutions to optimization problems that simultaneously minimize an entire class of objective functions. It is straightforward to show empirically that such solutions do not exist for most optimization problems. However, little is known on why this is the case and whether a characterization exists of problems that do have such solutions. In this article, we answer these questions by linking the existence of solutions that simultaneously optimize the class of Schur-convex functions, called least majorized elements, to (bi)submodular functions and the corresponding polyhedra. For this, we introduce a generalization of majorization and least majorized elements, called $(a,b)$-majorization and least $(a,b)$-majorized elements, and characterize the feasible sets of problems that have such elements in terms of these polyhedra. Hereby, we also obtain new characterizations of base and bisubmodular polyhedra that extend classical characterizations of these sets in terms of optimal greedy algorithms for linear optimization from the 1970s. We discuss the implications of our results for RAPs in power management applications and use the results to derive a new characterization of convex cooperative games and new properties of optimal estimators of specific regularized regression problems. In general, our results highlight the combinatorial nature of simultaneously optimizing solutions and, at the same time, provide a theoretical explanation for the observation that such solutions generally do not exist.
\end{abstract}

\section{Introduction}

\subsection{Resource allocation and simultaneous optimization}

Efficient resource allocation is an important challenge for many modern computing systems. The general goal in such resource allocation problems (RAPs) is to divide the available resource over individual users or system components as to maximize the overall system utility of this division (or, equivalently, to minimize the overall system cost). These problems occur in many different engineering applications such as telecommunications \cite{Palomar2005,Shams2014}, processor scheduling \cite{Gerards2016}, and regularized learning \cite{Dai2006,Mairal2011} (see also the overview in \cite{Patriksson2008}).

For many allocation problems, there has been interest in the development of algorithms that find allocations with good optimality guarantees that hold for multiple utility or cost functions simultaneously. One motivation for this is that different definitions of the cost of an allocation may exist that are, ideally, minimized simultaneously. For instance, in the context of smart storage systems in distribution grids, it is crucial that the system is operated so that the stress put on the grid is reduced as much as possible. This stress can be reduced in several ways, for instance by minimizing the peak energy consumption or flattening the overall load profile as much as possible. This leads to different objectives and corresponding cost functions that are preferably optimized at the same time (see also \cite{SchootUiterkamp2021}). 

Another motivation is that it is often hard in practice to specify the utility or cost function exactly due to, e.g., missing data or the absence of objective measures for such a function.  Instead, usually only some structural properties such as concavity or symmetry can be observed or assumed with reasonable certainty. Therefore, it is desirable to have a solution procedure that provides good solutions regardless of the exact description of the objective function. One example of such an application is speed scaling, where the goal is to schedule tasks on a processor and determine execution speeds for these tasks while minimizing the total energy usage of the processor and respecting any restrictions on the execution of the tasks such as deadlines. A common observation in this area is that the total energy usage depends primarily on the chosen speeds and that this dependence is convex. However, the exact relation between these two quantities depends on the specific properties of the used processor and is therefore generally unknown (see also \cite{Gerards2016}).

Most research on finding allocations with good performance guarantees for whole classes of    objective    functions focus on approximate solutions rather than optimal ones (see, e.g., \cite{Goel2006}) The main reason for this is that for many problems of interest there is no solution that simultaneously optimizes whole classes of utility or cost functions. However, little is known about why this is the case and whether there exists a characterization of problems that do have such a solution. These questions are the main motivation for the research in this article.

\subsection{Simultaneous optimization, majorization, and (bi)submodularity}

The existence of solutions that minimize whole classes of objective functions is closely related to the concept of majorization \cite{Marshall2011}. Majorization is a formalization of the vague notion that the components of one vector are less spread out (or closer together) than those of another vector. We say that, in this case, the former vector is \emph{majorized} by the latter. Majorization can be characterized in several ways, one of which is via Schur-convex functions. More precisely, given two vectors $x,y \in \mathbb{R}^n$,  $x$ is majorized by $y$ if and only if $\Phi(x) \leq \Phi(y)$ for all Schur-convex functions $\Phi$. Several variants of majorization exist that can also be characterized in terms of better objective values for particular classes of Schur-convex functions. Examples of these are weak submajorization, weak supermajorization, and weak absolute majorization, which are characterized by better objective values for non-decreasing, non-increasing, and element-wise even Schur-convex functions, respectively (see also \cite{Marshall2011} and Section~\ref{sec_maj_basic}).

The class of Schur-convex functions is broad and includes, e.g., symmetric quasi-convex functions. In particular, many common utility and cost functions in RAPs are special cases of Schur-convex functions. As a consequence, majorization and its characterization in terms of Schur-convex functions have many applications in, e.g., telecommunications \cite{Jorswiek2006} and economics \cite{Arnold2018} (see also \cite{Marshall2011}).

   Solutions that are majorized by all other vectors in a given feasible set, and thereby also simultaneously minimize all Schur-convex functions over this set, are known as least majorized elements.    The existence of such elements has hardly been investigated in the literature. In fact, the only non-trivial sets for which the existence of such elements has been established are (submodular) base polyhedra \cite{Dutta1989} (see also the historical overview in \cite{Frank2020}).

Several other connections exist between majorization    and (bi)submodularity.    For instance, majorization can be characterized in terms of submodular functions (see, e.g., page~44 of \cite{Fujishige2005}) and a similar characterization of weak absolute majorization exists in terms of bisubmodular functions \cite{Zhan2005}. Furthermore, several combinatorial sets that are in some way related to base polyhedra have least weakly sub- or supermajorized elements, i.e., vectors that are weakly sub- or supermajorized by all other vectors in the given set. Examples of such sets are bounded generalized polymatroids \cite{Tamir1995} and jump systems \cite{Ando1996_jump}.

Majorization and least majorized elements are also important concepts in economical theories of fair allocation of resources (see, e.g., \cite{Arnold2018} for an overview). In particular, they play an important role in convex cooperative games, which have the unique property that their core, i.e., the set of all payoffs where no group of players has an incentive to form their own competitive coalition, is a base polyhedron (see also \cite{Ichiishi1981}). Therefore, several key properties of submodular functions and base polyhedra can be proven using tools from cooperative game theory. In particular, existence results for particular solution vectors in the core of convex games translate directly into existence results of such vectors for base polyhedra. Examples of these are the egalitarian solution \cite{Dutta1989} and min-max fair and utilitarian solutions for uniform utility Nash and nonsymmetric bargaining games \cite{Chakrabarty2014}    (we come back to this in Section~\ref{sec_char}).

\subsection{Contributions}

A natural question is whether the existence of least (weakly absolutely) majorized and least weakly sub-or supermajorized elements is unique for base polyhedra and related sets. As discussed in the previous subsection, this directly relates to the question to what extent a characterization exists of optimization problems and solutions to these problems that simultaneously optimize multiple    (Schur-convex)    objectives. In this article, we provide an answer    to    these questions and, in the process, reveal a strong relation between majorization and base, submodular, supermodular, and bisubmodular polyhedra.

If we consider the ``classical" concept of majorization as discussed in the previous subsection, the answer to the first question is negative. For instance, given $R \in \mathbb{R}$, the vector $R(\frac{1}{n},\ldots,\frac{1}{n})$ is a least majorized element of any subset of $\mathbb{R}^n$ that contains it and whose vectors have the same element-sum $R$. This observation follows directly from Jensen's inequality. Similarly, the zero vector is a least weakly absolutely majorized element of any subset of $\mathbb{R}^n$ that contains it. These observations show that sets admitting least majorized elements are not necessarily ``nice", i.e., polyhedral or convex, and thus also not necessarily base polyhedra. Therefore, it seems unlikely that a characterization can be given of sets that have least majorized elements.

These observations motivate us to introduce a more general notion of majorization that we call \emph{$(a,b)$-majorization}. This concept can be interpreted as majorization relative to a positive scaling vector $a$ and a shifting vector $b$ and is a slight generalization of the concept of $d$-majorization in \cite{Veinott1971}. In terms of optimization, $(a,b)$-majorization is concerned with order-preserving objective functions of the form $\sum_{i \in N} a_i \phi (\frac{x_i + b_i}{a_i}  )$, where $\phi$ is a continuous convex function.    Such objective functions naturally occur in energy scheduling applications such as power allocation for multi-channel communication systems and speed scaling for computer processors \cite{SchootUiterkamp2021} and price mechanism design \cite{Reijnders2022}.   Analogously to ``classical" majorization, we say that $x$ is $(a,b)$-majorized by $y$ if $\sum_{i \in N} a_i \phi (\frac{x_i + b_i}{a_i}) \leq  \sum_{i \in N} a_i \phi (\frac{y_i + b_i}{a_i} )$ for all continuous convex functions $\phi$. Moreover, given a set $C \subseteq \mathbb{R}^n$, we say that $x^* \in C$ is a least $(a,b)$-majorized element of $C$ if it is an optimal solution to the problem $\min_{x \in C} \sum_{i \in N} a_i \phi ( \frac{x_i + b_i}{a_i}  )$ for any choice of continuous convex function $\phi$. Analogously to ``classical" majorization, we also define similar generalizations of weak submajorization, weak supermajorization, and weak absolute majorization and of the corresponding least majorized elements    (see Section~\ref{sec_maj_gen}).

This generalization of majorization allows us to obtain a positive answer to the questions posed at the start of this subsection. For this, we build upon a result in earlier work \cite{SchootUiterkamp2021} that shows the existence of least $(a,b)$-majorized elements for base polyhedra. As a first contribution, in the present article, we extend this result by establishing the existence of least weakly $(a,b)$-sub- and $(a,b)$-supermajorized and least weakly absolutely $(a,b)$-majorized elements for several related sets including submodular, supermodular, and bisubmodular polyhedra. Thereby, we generalize some results in \cite{Dutta1989,Tamir1995} on the existence of least weakly sub- and supermajorized elements for special cases of these polyhedra. In fact, for the case of weak absolute majorization, no such existence results were known for these sets even for the case of ``classical" majorization (i.e., when $a=1$ and $b=0$) and we are the first to establish them.

As our second and main contribution, we show that the existence of least $(a,b)$-majorized elements for all choices of scaling and shifting vectors $a$ and $b$ is a unique property of base polyhedra within the class of compact convex sets. Hereby, we prove a reverse of the earlier mentioned result in \cite{SchootUiterkamp2021}. Moreover, still within the class of compact convex sets, we show that the existence of least weakly absolutely $(a,b)$-majorized elements for all choices of $a$ and $b$ is a unique property of bisubmodular polyhedra. Thus, we obtain new characterizations of base and bisubmodular polyhedra in terms of (weak absolute) $(a,b)$-majorization. Additionally, we completely characterize sets admitting least weakly $(a,b)$-sub- or $(a,b)$-supermajorized elements for all choices of $a$ and $b$ as those that are contained in a super- or submodular polyhedron, respectively, and contain the corresponding base polyhedron.

For $a=1$, all these results carry over to the integral versions of these sets, i.e., the collection of integral points in the set. Here, the class of compact convex sets is replaced by that of bounded integral hole-free sets, i.e., bounded sets that contain exactly all integral points in their convex hull. All characterization results for both set versions are summarized in Theorems~\ref{th_super}-\ref{th_bi_int}.

The existence results are proven based on the earlier mentioned result for base polyhedra in earlier work \cite{SchootUiterkamp2021} and on the relation of the aforementioned sets with these polyhedra. For instance, the existence of least weakly $(a,b)$-supermajorized elements for submodular polyhedra can be proven using the fact that such polyhedra contain vectors that are in some sense maximal and thus necessarily minimize any non-increasing function over such polyhedra. These maximal vectors form a base polyhedron within the given submodular polyhedron. Furthermore, the existence of least weakly absolutely $(a,b)$-majorized elements for bisubmodular polyhedra uses the fact that such polyhedra are the intersection of (reflections of) particular submodular polyhedra and contain all corresponding (reflected) base polyhedra (see also \cite{Ando1996,Fujishige2005}).

The proof of the characterization results for base and bisubmodular polyhedra is inspired by classical characterizations for the optimality of greedy algorithms for linear optimization over these polyhedra \cite{Edmonds2003,Dunstan1973,Nakamura1988}. The latter characterizations state that for each permutation of the index set there exists a vector that simultaneously minimizes any linear cost function over these polyhedra whose coefficients are monotonically increasing or decreasing under this permutation. In this light, our characterization results can be seen as an extension of these classical results to    (Schur-)convex    objective functions.

We demonstrate the impact of our results in three different fields. First, we focus on RAPs in power management and in particular on energy storage scheduling problems. In these problems, the utility or cost of an allocation depends on the actual allocation plus a fixed (uncontrollable) load such as static energy consumption    (see also \cite{vanderKlauw2017}). Our results show that RAPs over base polyhedra are the only convex (concave) RAPs where for any given fixed load, there exists an allocation that simultaneously optimizes any symmetric (quasi-)convex ((quasi-)concave) objective function of the combined load and allocation.    This work therefore provides a theoretical explanation for the necessity of objective trade-offs for more general RAPs, even when those objectives do not conflict at first glance (e.g., when the considered objective functions are all symmetric utility functions). This motivates the necessity and development of multi-objective optimization algorithms for the operation of the corresponding systems.

The second application is in cooperative game theory. We obtain a new characterization of convex cooperative games with transferable utility in terms of the existence of egalitarian solutions. Although this characterization is simply a direct reformulation of one of our main results (Theorem~\ref{th_base}) in game-theoretical terms, no similar characterizations have been proposed before to the best of our knowledge.

Finally, we focus on regularized regression problems and their corresponding optimal estimators. We show that for certain regression problems arising in orthogonal design experiments, the optimal regression estimators for LASSO and max-norm regularization are least weakly absolutely $(a,b)$-majorized elements, while the optimal estimator for ridge regression is not. These observations might be a stepping stone to demonstrate the optimality of these estimators for other regression problems.

Summarizing, our concrete contributions are as follows:
\begin{enumerate}
\item
We introduce a natural generalization of majorization, called $(a,b)$-majorization, and investigate its properties and relation to ``classical" majorization;
\item
We completely characterize all compact convex sets and bounded integral hole-free sets that have least weakly $(a,b)$-sub- or $(a,b)$-supermajorized or least (weakly absolutely) $(a,b)$-majorized elements for all pairs of scaling vectors $a$ and shifting vectors $b$ in terms of submodular, supermodular, base, and bisubmodular polyhedra;
\item
We characterize base and bisubmodular polyhedra in terms of the existence of least $(a,b)$ majorized and least weakly absolutely $(a,b)$-majorized elements for all pairs of scaling vectors $a$ and shifting vectors $b$, respectively;
\item
On the application side, we provide a theoretical explanation for the presence of conflicting utility or cost objectives in more general RAPs, give a new characterization of convex cooperative games, and establish new properties of a specific class of regression estimators.
\end{enumerate}

The outline of the remainder of this article is as follows. In Section~\ref{sec_pre}, we introduce the necessary concepts and results regarding (bi)submodular functions and the related polyhedra and in Section~\ref{sec_maj}, we introduce $(a,b)$-majorization. In Section~\ref{sec_existence}, we prove the existence of least $(a,b)$-majorized and related elements for submodular, supermodular, and bisubmodular polyhedra and related sets and in Section~\ref{sec_from_maj_to_sub}, we derive a complete characterization of these elements in terms of these polyhedra. In Section~\ref{sec_appl}, we demonstrate the impact of these characterization result in several applications and in Section~\ref{sec_concl}, we state our conclusions and directions for future research.

\section{Preliminaries}
\label{sec_pre}

In this section, we introduce some general notation and all concepts and known results on (bi)submodularity that we require for this article. 

\subsection{General notation}

Throughout, we denote the index set by $N := \lbrace 1,\ldots,n \rbrace$, the power set of $N$ by $2^N := \lbrace S \ | \ S \subseteq N \rbrace$. Given a vector $x \in \mathbb{R}^n$ and a subset $S \in 2^N$, we denote the sum of all elements of $x$ whose indices are in $S$ by $x(S) := \sum_{i \in S} x_i$. We denote the convex hull of $C$ by $\text{co}(C)$. For convenience, we say that a given vector, set, or function is \emph{integral} if it is integer-valued. We call an integral set $C \subseteq \mathbb{Z}^n$ \emph{hole-free} if it    consists of    all integral points in its convex hull, i.e., if $C = \text{co}(C) \cap \mathbb{Z}^n$. Moreover, we say that a function $\Phi$ is continuous and convex on $C$ if there exists a continuous and convex function $\tilde{\Phi}$ on $\mathbb{R}^n$ that coincides with $\Phi$ on $C$. Given two vectors $x,y \in \mathbb{R}^n$, we write $x \leq y$ if $x_i \leq y_i$ for all $i \in N$ and $x < y$ if, in addition, we have $x_i < y_i$ for at least one index $i \in N$. For each $i \in N$, we denote by $e^i$ the unit base vector of dimension~$n$ corresponding to $i$, i.e., $e^i_k = 1$ if $k=i$ and $e^i_k =0 $ if $k \neq i$.

\subsection{Submodularity}
A set function $f{:} \ 2^N \rightarrow \mathbb{R}^n$ is \emph{submodular} if $f(A) + f(B) \geq f(A \cup B) + f(A \cap B)$ for all $A,B \in 2^N$, where we assume that $f(\emptyset) = 0$. Moreover, we say that a set function $g$ is \emph{supermodular} if $-g$ is submodular. Note that if $f$ is submodular, the function $f^{\#}(A) := f(N) - f(N \backslash A)$ is supermodular and we call $f^{\#}$ the \emph{dual supermodular function} of $f$. Analogously, we call the submodular function $g^{\#}(A) := g(N) - g(N \backslash A)$ the \emph{dual submodular function} of $g$.

Given a submodular function $f$, we denote by $P(f)$ and $B(f)$ the submodular and base polyhedron    defined    by $f$, which are respectively given by
\begin{align*}
    P(f) &:= \lbrace x \in \mathbb{R}^n \ | \ x(A) \leq f(A) \ \forall A \in 2^N \rbrace; \\
B(f) &:= \lbrace x \in P(f) \ | \ x(N) = f(N) \rbrace.
\end{align*}
Analogously, given a supermodular function $g$, we denote by $P_{\text{sup}}(g)$ and $B_{\text{sup}}(g)$ the supermodular and base  polyhedron    defined    by $g$, which are respectively given by
\begin{align*}
 P_{\text{sup}}(g) &:= \lbrace x \in \mathbb{R}^n \ | \ x(A) \geq g(A) \ \forall A \in 2^N \rbrace; \\
B_{\text{sup}}(g) &:= \lbrace x \in P_{\text{sup}}(g) \ | \ x(N) = g(N) \rbrace.
\end{align*}
Note, that $B(f) = B_{\text{sup}}(f^{\#})$ for any submodular function $f$ since $f(N) = f^{\#}(N)$ and since, for any $A \in 2^N$, $x \in B(f)$ implies $x(A) = f(N) - x(N \backslash A) \geq f(N) - f(N \backslash A) = f^{\#}(A)$ and $x \in B_{\text{sup}}(f^{\#})$ implies $x(A) = f^{\#}(N) - x(N \backslash A) \leq f^{\#}(N) - f^{\#}(N \backslash A) = f(A)$.

We also define the \emph{integral} submodular (supermodular) and base polyhedron    defined    by an integral submodular function $f$ (integral supermodular function $g$) as the set of all integer-valued points in $P(f)$, $B(f)$, $P_{\text{sup}}(g)$, and $B_{\text{sup}}(g)$, respectively:
\begin{align*}
    P^{\mathbb{Z}}(f) &:= P(f) \cap \mathbb{Z}^n; &
    B^{\mathbb{Z}}(f) &:= B(f) \cap \mathbb{Z}^n; \\
    P_{\text{sup}}^{\mathbb{Z}}(g) &:= P_{\text{sup}}(g) \cap \mathbb{Z}^n; &
    B_{\text{sup}}^{\mathbb{Z}}(g) &:= B_{\text{sup}}(g) \cap \mathbb{Z}^n.
\end{align*}
Note that all these sets are hole-free.

Given a submodular function $f$ and a vector $v \in \mathbb{R}^n$, we define the \emph{reduction} of $f$ by $v$ as the set function
\begin{equation*}
f^v (A) := \min(f(Z) + v(A \backslash Z) \ | \ Z \subseteq A ), \quad A \in 2^N.
\end{equation*}
It follows from, e.g., \cite[Theorem 3.3]{Fujishige2005} that also $f^v$ is submodular and that 
\begin{align*}
P(f^v) &= \lbrace x \in P(f) \ | \ x \leq v \rbrace; &
P^{\mathbb{Z}}(f^v) &= \lbrace x \in P^{\mathbb{Z}}(f) \ | \ x \leq v \rbrace. 
\end{align*}
Analogously, given a supermodular function $g$, the reduction of $g$ by $v$ is defined as
\begin{equation*}
g^v (A) := \max(g(Z) + v(A \backslash Z) \ | \ Z \subseteq A ), \quad A \in 2^N.
\end{equation*}
Moreover, $g^v$ is supermodular and 
\begin{align*}
P_{\text{sup}}(g^v) &= \lbrace x \in P_{\text{sup}}(g) \ | \ x \geq v \rbrace; &
P_{\text{sup}}^{\mathbb{Z}}(g^v) &= \lbrace x \in P_{\text{sup}}^{\mathbb{Z}}(g) \ | \ x \geq v \rbrace. 
\end{align*}

Lemma~\ref{lemma_ineq} shows that we can always increase (decrease) the entries of a vector in a submodular (supermodular) polyhedron in such a way that it becomes an element of the corresponding base polyhedron:
\begin{lemma}[e.g., Theorem 2.3 in \cite{Fujishige2005}]
Let a (integral) submodular function $f$ and a (integral) supermodular function $g$ be given. Then the following hold:
\begin{itemize}
\item
For each $x \in P(f)$ ($x \in P^{\mathbb{Z}}(f)$) there exists $y \in B(f)$ ($y \in B^{\mathbb{Z}}(f)$) such that $x \leq y$;
\item
For each $x \in P_{\text{sup}}(g)$ ($x \in P_{\text{sup}}^{\mathbb{Z}}(g)$) there exists $y \in B_{\text{sup}}(g)$ ($y \in B_{\text{sup}}^{\mathbb{Z}}(g)$) such that $x \geq y$.
\end{itemize}
\label{lemma_ineq}
\end{lemma}

Finally, in Lemma~\ref{lemma_extreme}, we state the known result that the extreme points of a base polyhedron can be expressed analytically in terms of the submodular function    defining    the base polyhedron:
\begin{lemma}[e.g., Theorem 3.4 of \cite{Fujishige1983}]
For each extreme point $z$ of a (integral) base polyhedron    defined    by a (integral) submodular function $f$, there exists a permutation $\pi$ of $N$ such that $z_{\pi(1)} = f(\lbrace \pi(1) \rbrace)$ and $z_{\pi(k)} = f(\lbrace \pi(1),\ldots,\pi(k) \rbrace) - f(\lbrace \pi(1),\ldots,\pi(k-1) \rbrace$ for each $k \in \lbrace 2,\ldots, n \rbrace$.
\label{lemma_extreme}
\end{lemma}
For more background on submodular functions in general and the above lemmas in particular, we refer to \cite{Fujishige2005}.

\subsection{Bisubmodularity}
\label{sec_bisubmodular}

We denote the set of all ordered pairs $(S,T)$ of disjoint subsets of $N$ by $3^N := \lbrace (S,T) \ | \ S,T \subseteq N, \ S \cap T = \emptyset \rbrace$. Given an element $U = (S,T) \in 3^N$, we use the notation $U^+$ to denote $S$ and $U^-$ to denote $T$. Moreover, for any $X \subseteq N$    and vector $s \in \lbrace -1,1 \rbrace^n$,    we use the notation $X \ | \ s$ to denote the biset $(\lbrace i \in X \ | \ s_i = 1 \rbrace, \lbrace i \in X \ | \ s_i = -1 \rbrace)$. We use the symbols $\sqsubseteq$ and $\sqsupseteq$ to define the subset and superset    relations    on two elements of $3^N$, i.e., for any two pairs $(S_1,T_1),(S_2,T_2) \in 3^N$, we write $(S_1,T_1) \sqsubseteq (S_2,T_2)$ if $S_1 \subseteq S_2$ and $T_1 \subseteq T_2$ and we have $(S_1,T_1) \sqsupseteq (S_2,T_2)$ if $S_1 \supseteq S_2$ and $T_1 \supseteq T_2$. Moreover, we define the reduced union and intersection of $(S_1,T_1)$ and $(S_2,T_2)$ as
\begin{align*}
(S_1,T_1) \sqcup (S_2,T_2) &:= ((S_1 \cup S_2) \backslash (T_1 \cup T_2), (T_1 \cup T_2) \backslash (S_1 \cup S_2)); \\
(S_1,T_1) \sqcap (S_2,T_2) &:= (S_1 \cap S_2, T_1 \cap T_2),
\end{align*}
which can be seen as a generalization of the usual union and intersection operations, respectively.

A biset function $h{:} \ 3^N \rightarrow \mathbb{R}$ with $h(\emptyset,\emptyset) := 0$ is \emph{bisubmodular} if for each two pairs $(S_1,T_1),(S_2,T_2) \in 3^N$ we have
$$h(S_1,T_1) + h(S_2,T_2) \geq h((S_1,T_1) \sqcup (S_2,T_2)) + h((S_1,T_1) \sqcap (S_2,T_2)).$$ 
Given a bisubmodular function $h$, the \emph{bisubmodular polyhedron}    defined    by $h$ is the set
\begin{equation*}
\tilde{B}(h) := \lbrace x \in \mathbb{R}^n \ | \ x(S) - x(T) \leq h(S,T) \ \forall (S,T) \in 3^N \rbrace.
\end{equation*}
Analogously to submodular functions and polyhedra, we define the \emph{integral} bisubmodular polyhedron    defined    by an integral bisubmodular function $h$ as the set of all integral points in $\tilde{B}(h)$, i.e., $\tilde{B}^{\mathbb{Z}}(h) := \tilde{B}(h) \cap \mathbb{Z}^n$. Note that $\tilde{B}^{\mathbb{Z}}(h)$ is hole-free.

Given a pair $(S,T) \in 3^N$ and a bisubmodular function $h$, we denote by $2^{(S,T)}$ the set of all pairs $(X,Y)$ with $(X,Y) \sqsubseteq (S,T)$ and we define the set function $h_{(S,T)}(X) := h(S \cap X, T \cap X)$. Note that $h_{(S,T)}$ is submodular since we have the following for any $X,Y \in 2^{N}$:
\begin{align*}
& \quad \ h_{(S,T)}(X) + h_{(S,T)}(Y) \\
&=
h(S \cap X, T \cap X) + h(S \cap Y, T \cap Y) \\
& \geq h((S \cap X, T \cap X) \sqcup (S \cap Y,T \cap Y)) + h((S \cap X, T \cap X) \sqcap (S \cap Y,T \cap Y)) \\
&= h( S \cap (X \cup Y), T \cap (X \cup Y)) + h(S \cap X \cap Y, T \cap X \cap Y) \\
&= h_{(S,T)}(X \cup Y) + h_{(S,T)}(X \cap Y).
\end{align*}
If $S \cup T = N$, we call $(S,T)$ an \emph{orthant} of $\mathbb{R}^n$. We define the submodular and base polyhedron in the orthant $(S,T)$ as
\begin{align*}
P_{(S,T)}(h) &:= \lbrace x \in \mathbb{R}^n \ | \ x(X) - x(Y) \leq h(X,Y) \ \forall (X,Y) \in 2^{(S,T)} \rbrace; \\
B_{(S,T)}(h) & := \lbrace x \in P_{(S,T)}(h) \ | \ x(S) - x(T) = h(S,T) \rbrace.
\end{align*}
Note that, by definition of $h_{(S,T)}$, we have for any orthant $(S,T)$ of $\mathbb{R}^n$ that
\begin{equation*}
P_{(S,T)}(h) := \left\{ x \in \mathbb{R}^n \ | \ \exists y \in P(h_{(S,T)}) \text{ s.t. } x_i = \begin{cases} y_i & \text{if } i \in S; \\
-y_i & \text{if } i \in T. \end{cases} \right\} 
\end{equation*}
This implies for any $x \in \mathbb{R}^n$ that $x \in P_{(S,T)}(h)$ ($x \in B_{(S,T)}(h)$) if and only if $\tilde{x} \in P(h_{(S,T)})$ ($\tilde{x} \in B(h_{(S,T)})$) where $\tilde{x}_i = x_i$ if $i \in S$ and $\tilde{x}_i = -x_i$ if $i \in T$.    Note that $\tilde{B}(h)$ is the intersection of all sets $P_{(S,T)}(h)$ for all orthants $(S,T)$ of $\mathbb{R}^n$. Moreover, for any such orthant, the corresponding base polyhedron $B_{(S,T)}(h)$ is contained in $\tilde{B}(h)$ (see, e.g., \cite[Lemma 3.60]{Fujishige2005}).

Finally, analogously to base polyhedra and Lemma~\ref{lemma_extreme}, the extreme points of a bisubmodular polyhedron can be expressed analytically in terms of the underlying bisubmodular function:
\begin{lemma}[e.g., Theorem 3.4 in \cite{Ando1996}]
For each extreme point $z$ of a bisubmodular polyhedron    defined    by a bisubmodular function $h$, there exists a permutation $\pi$ of $N$ and a sign vector $s \in \lbrace -1,1 \rbrace^n$ such that $z_{\pi(1)} = s_{\pi(1)} h(\lbrace \pi(1) \rbrace \ | \ s)$ and $z_{\pi(k)} = s_{\pi(k)} (h(\lbrace \pi(1),\ldots,\pi(k) \rbrace \ | \ s) - h(\lbrace \pi(1),\ldots,\pi(k-1) \rbrace \ | \ s))$ for each $k \in \lbrace 2, \ldots, n\rbrace$.
\label{lemma_extreme_bi}
\end{lemma}

For more background on bisubmodular functions and polyhedra and on their discussed properties, we refer to \cite{Bouchet1995,Ando1996,Fujishige1997}.

\section{Majorization and $(a,b)$-majorization}
\label{sec_maj}

In this section, we discuss the classical concept of majorization (Section~\ref{sec_maj_basic}) and introduce a more general version of this concept that we call \emph{$(a,b)$-majorization} (Section~\ref{sec_maj_gen}).

\subsection{Classical majorization}
\label{sec_maj_basic}

Given a vector $x \in \mathbb{R}^n$, we denote the non-increasing ordering of the elements in $x$ by $x^{\downarrow}$ (so that $x^{\downarrow}_1 \geq \ldots \geq x^{\downarrow}_n$) and the non-decreasing order by $x^{\uparrow}$ (so that $x^{\uparrow}_1 \leq \ldots \leq x^{\uparrow}_n$). Given also another vector $y \in \mathbb{R}^n$, we say that $x$ is 
\begin{itemize}
\item
weakly submajorized by $y$, denoted by $x \prec^{\downarrow} y$, if $\sum_{i=1}^k x^{\downarrow}_i \leq \sum_{i =1}^k y^{\downarrow}_i$ for all $k \in N$;
\item
weakly supermajorized by $y$, denoted by $x \prec^{\uparrow} y$, if $\sum_{i=1}^k x^{\uparrow}_i \geq \sum_{i =1}^k y^{\uparrow}_i$ for all $k \in N$;
\item majorized by $y$, denoted by $x \prec y$, if $x$ is both weakly sub- and supermajorized by $y$ or, equivalently, $x \prec^{\downarrow} y$ and $x(N) = y(N)$.
\end{itemize}

Several characterizations of these majorization concepts exist. One characterization that is particularly useful in the context of optimization problems is in terms of objective values for Schur-convex functions. A function $\Phi{:} \ \mathbb{R}^n \rightarrow \mathbb{R}$ is said to be Schur-convex if it preserves the order of majorization, i.e., if $\Phi(x) \leq \Phi(y)$ whenever $x \prec y$. The following well-known characterizations of majorization and Schur-convex functions will be useful in the remainder of this article:
\begin{lemma}
Given vectors $x,y \in \mathbb{R}^n$, the following statements are equivalent:
\begin{enumerate}
\item
$x \prec^{\downarrow} y$;
\item
$\Phi(x) \leq \Phi(y)$ for all non-decreasing continuous Schur-convex functions $\Phi$;
\item
$\sum_{i \in N} \phi(x_i) \leq \sum_{i \in N} \phi(y_i)$ for all non-decreasing continuous convex functions $\phi$;
\item
$\sum_{i \in N} \max(0, \alpha + x_i) \leq \sum_{i \in N} \max(0, \alpha + y_i)$ for all $\alpha \in \mathbb{R}$. 
\end{enumerate}
\label{lemma_maj_sub}
\end{lemma}
\begin{lemma}
Given vectors $x,y \in \mathbb{R}^n$, the following statements are equivalent:
\begin{enumerate}
\item
$x \prec^{\uparrow} y$;
\item
$\Phi(x) \leq \Phi(y)$ for all non-increasing continuous Schur-convex functions $\Phi$;
\item
$\sum_{i \in N} \phi(x_i) \leq \sum_{i \in N} \phi(y_i)$ for all non-increasing continuous convex functions $\phi$;
\item
$\sum_{i \in N} \max(0, \alpha - x_i) \leq \sum_{i \in N} \max(0, \alpha - y_i)$ for all $\alpha \in \mathbb{R}$. 
\end{enumerate}
\label{lemma_maj_sup}
\end{lemma}
\begin{lemma}
Given vectors $x,y \in \mathbb{R}^n$, the following statements are equivalent:
\begin{enumerate}
\item
$x \prec y$;
\item
$\Phi(x) \leq \Phi(y)$ for all continuous Schur-convex functions $\Phi$;
\item
$\sum_{i \in N} \phi(x_i) \leq \sum_{i \in N} \phi(y_i)$ for all continuous convex functions $\phi$;
\item
$x(N) = y(N)$ and $\sum_{i \in N} \max(0, \alpha - x_i) \leq \sum_{i \in N} \max(0, \alpha - y_i)$ for all $\alpha \in \mathbb{R}$. 
\end{enumerate}
\label{lemma_maj}
\end{lemma}
We refer to Propositions 1.A.2, 1.A.8, and 4.B.1-4.B.4 in \cite{Marshall2011} for the proofs of these results and to \cite{Marshall2011} in general for more background on majorization.

Alongside these well-known majorization concepts, we also consider the concept of weak \emph{absolute} majorization. Given a vector $x \in \mathbb{R}^n$, we denote the non-increasing ordering of \emph{the absolute values of} the elements of $x$ by $x^{\text{abs}}$. We say that $x$ is \emph{weakly absolutely} majorized by $y$, denoted by $x \prec^{\text{abs}} y$, if $\sum_{i=1}^k |x^{\text{abs}}_i| \leq \sum_{i =1}^k |y^{\text{abs}}_i|$ for all $k \in N$ (see also \cite{Eaton1984,Zhan2005,Niezgoda2012}). Although this majorization concept is less known than the three other concepts discussed above, it occurs regularly as an illustrative example when studying majorization from the perspective of ordered groups (see, e.g., \cite{Francis2014}).

Lemma~\ref{lemma_maj_abs} provides a characterization of weak absolute majorization in terms of a particular class of Schur-convex functions that we call \emph{monotonically even}. This class consists of all Schur-convex functions $\Phi$ for which 
\begin{itemize}
\item
$\Phi(x_1,\ldots,x_{i-1},-x_{i},x_{i+1},\ldots,x_n) = \Phi(x_1,\ldots,x_{i-1},x_{i},x_{i+1},\ldots,x_n)$ for all $x \in \mathbb{R}^n$ and $i \in N$ ($\Phi$ is even);
\item
$\Phi(x_1,\ldots,x_{i-1},x_{i},x_{i+1},\ldots,x_n) \leq \Phi(x_1,\ldots,x_{i-1},x_{i} - \epsilon,x_{i+1},\ldots,x_n)$ for all $\epsilon > 0$ and $x \in \mathbb{R}^n$ with $x_i \leq 0$ ($\Phi$ is element-wise non-increasing for negative inputs);
\item
$\Phi(x_1,\ldots,x_{i-1},x_{i},x_{i+1},\ldots,x_n) \leq \Phi(x_1,\ldots,x_{i-1},x_{i} + \epsilon,x_{i+1},\ldots,x_n)$ for all $\epsilon > 0$ and $x \in \mathbb{R}^n$ with $x_i \geq 0$ ($\Phi$ is element-wise non-decreasing for positive inputs).
\end{itemize}
This    definition    directly leads to a characterization in terms of separable \emph{even} convex objective functions, i.e., functions of the form $\sum_{i \in N} \phi(x_i)$ where each $\phi$ is convex and $\phi(-\zeta) = \phi(\zeta)$ for all $\zeta \in \mathbb{R}$ (note that evenness and convexity of $\phi$ directly ensures that $\phi$ is also monotonically even). This characterization and its proof are analogous to those corresponding to Lemmas~\ref{lemma_maj_sub}-\ref{lemma_maj}:
\begin{lemma}
Given vectors $x,y \in \mathbb{R}^n$, the following statements are equivalent:
\begin{enumerate}
\item
$x \prec^{\text{abs}} y$;
\item
$\Phi(x) \leq \Phi(y)$ for all monotonically even continuous Schur-convex functions $\Phi$;
\item
$\sum_{i \in N} \phi(x_i) \leq \sum_{i \in N} \phi(y_i)$ for all even continuous convex functions $\phi$;
\item
$\sum_{i \in N} \max(\alpha - x_i,0, \alpha + x_i) \leq \sum_{i \in N} \max(\alpha - y_i,0, \alpha + y_i)$ for all $\alpha \in \mathbb{R}$. 
\end{enumerate}
\label{lemma_maj_abs}
\end{lemma}
\begin{proof}
  
Proof of (1) $\Rightarrow$ (2): Given a monotonically even Schur-convex function $\Phi$, consider the non-decreasing Schur-convex function
\begin{equation*}
\tilde{\Phi}(x) = \begin{cases} \Phi(0) & \text{if } x \not\geq 0, \\
\Phi(x) & \text{if } x \geq 0. \end{cases}
\end{equation*}
 Let $\tilde{x} := (|x_i|)_{i \in N}$ and $\tilde{y} := (|y_i|)_{i \in N}$. If $x \prec^{\text{abs}} y$, then $\tilde{x} \prec^{\downarrow} \tilde{y}$ and thus $\tilde{\Phi}(\tilde{x}) \leq \tilde{\Phi}(\tilde{y})$. Since $\tilde{x}$ and $\tilde{y}$ are non-negative, it follows that $\Phi(\tilde{x}) \leq \Phi(\tilde{y})$ and, because $\Phi$ is even, that $\Phi(x) \leq \Phi(y)$.

Proof of (2) $\Rightarrow$ (3): Given a continuous convex function $\phi$, the function $\sum_{i \in N} \phi(x_i)$ is Schur-convex. If $\phi$ is also even, then $\sum_{i \in N} \phi(x_i)$ is Schur-convex and monotonically even and thus (3) is a special case of (2).

Proof of (3) $\Rightarrow$ (4): Note that for any $\alpha \in \mathbb{R}$ the function $ \max(\alpha - z,0,\alpha + z)$ is even and convex. Thus, (4) is a special case of (3).

Proof of (4) $\Rightarrow$ (1): We define for a given $k \in N$ the even convex function $\tilde{\phi}^k (z) := \max(-z - |y^{\text{abs}}_k|,0,z - |y^{\text{abs}}_k|)$. Then we have $\sum_{i=1}^k \tilde{\phi}^k(y^{\text{abs}}_i) = \sum_{i=1}^k |y^{\text{abs}}_i| - k|y^{\text{abs}}_k|$ and $\sum_{i=k+1}^n \tilde{\phi}^k(y^{\text{abs}}_i) = 0$. Note that, by assumption, we have $\sum_{i \in N} \tilde{\phi}^k(x_i) \leq \sum_{i \in N} \tilde{\phi}^k(y_i)$. Moreover, for any $z \in \mathbb{R}$, we have $\tilde{\phi}^k (z) \geq 0$ and $\tilde{\phi}^k (z) \geq |z| - |y^{\text{abs}}_k|$. It follows that
\begin{align*}
\sum_{i=1}^k |y^{\text{abs}}_i| - k|y^{\text{abs}}_k|
&=
\sum_{i \in N} \tilde{\phi}^k (y_i)
\geq \sum_{i \in N} \tilde{\phi}^k (x_i) \\
&\geq \sum_{i =1}^k \tilde{\phi}^k (x^{\text{abs}}_i)
\geq \sum_{i=1}^k |x^{\text{abs}}_i| - k|y^{\text{abs}}_k|,
\end{align*}
which implies that $\sum_{i=1}^k |x^{\text{abs}}_i| \leq \sum_{i=1}^k |y^{\text{abs}}_i|$. Since this holds for all $k \in N$, it follows that $x \prec^{\text{abs}} y$.
\end{proof}

\subsection{$(a,b)$-majorization}
\label{sec_maj_gen}

We introduce below a more general version of all four majorization concepts discussed in Section~\ref{sec_maj_basic}. This generalization, which we call \emph{$(a,b)$-majorization}, can be interpreted as majorization relative to scaling by a (positive) vector $a$ and shifting by a vector $b$. More precisely, for fixed scaling and shifting vectors $a \in \mathbb{R}^n_{>0}$ and $b \in \mathbb{R}^n$ and given two vectors $x,y \in \mathbb{R}^n$, we say that $x$ is 
\begin{itemize}
\item
\sloppy $(a,b)$-majorized by $y$, denoted by $x \prec_{(a,b)} y$, if $x(N) = y(N)$ and for all continuous convex functions $\phi$ we have $\sum_{i \in N} a_i \phi (\frac{x_i + b_i}{a_i}  ) \leq \sum_{i \in N} a_i \phi (\frac{y_i + b_i}{a_i}  )$;
\item
weakly $(a,b)$-submajorized by $y$, denoted by $x \prec_{(a,b)}^{\downarrow} y$, if for all non-decreasing continuous convex functions $\phi$ we have $\sum_{i \in N} a_i \phi (\frac{x_i + b_i}{a_i}  ) \leq \sum_{i \in N} a_i \phi (\frac{y_i + b_i}{a_i}  )$;
\item
weakly $(a,b)$-supermajorized by $y$, denoted by $x \prec_{(a,b)}^{\uparrow} y$, if for all non-increasing continuous convex functions $\phi$ we have $\sum_{i \in N} a_i \phi (\frac{x_i + b_i}{a_i}  ) \leq \sum_{i \in N} a_i \phi (\frac{y_i + b_i}{a_i}  )$;
\item
weakly absolutely $(a,b)$-majorized by $y$, denoted by $x \prec_{(a,b)}^{\text{abs}} y$, if for all even continuous convex functions $\phi$ we have $\sum_{i \in N} a_i \phi (\frac{x_i+ b_i}{a_i}  ) \leq \sum_{i \in N} a_i \phi (\frac{y_i + b_i}{a_i} )$.
\end{itemize}
Moreover, given a set $C \subseteq \mathbb{R}^n$, a vector $x \in C$ is a 
\begin{itemize}
\item
\emph{least $(a,b)$-majorized} element of $C$ if $x \prec_{(a,b)} y$ for all $y \in C$;
\item
\emph{least weakly $(a,b)$-submajorized} element of $C$ if $x \prec_{(a,b)}^{\downarrow} y$ for all $y \in C$;
\item
\emph{least weakly $(a,b)$-supermajorized} element of $C$ if $x \prec_{(a,b)}^{\uparrow} y$ for all $y \in C$;
\item
\emph{least weakly absolutely $(a,b)$-majorized} element of $C$ if $x \prec_{(a,b)}^{\text{abs}} y$ for all $y \in C$.
\end{itemize}
For convenience, if the specific type of majorized element is clear from the context or not relevant, we call all these four types of elements least $(a,b)$-majorized elements. If also the vectors $a$ and $b$ are not specified, we call them least majorized elements.

The concept of $(a,b)$-majorization is a slight generalization of the earlier mentioned concept of $d$-majorization \cite{Veinott1971}, i.e., $d$-majorization is equivalent to $(d,0)$-majorization. Although the inclusion of a shift vector $b$ may not seem as a significant generalization, it is a crucial factor when proving our main characterization results in Section~\ref{sec_from_maj_to_sub} with regard to submodular, supermodular, base, and bisubmodular polyhedra.

Note that, as opposed to the ``classical" majorization concepts of Section~\ref{sec_maj_basic}, a definition of $(a,b)$-majorization cannot be given in terms of partial sums of non-increasing, non-decreasing, or absolute orders. Similar generalizations of majorization such as $d$-majorization \cite{Veinott1971} are introduced via an alternative definition of majorization involving doubly stochastic matrices (see also Chapter~2 and Section~14.B of \cite{Marshall2011}) or in terms of optimality for particular classes of separable convex objective functions (see, e.g., \cite{Joe1990}). Since we focus in this article on optimization problems, we follow the latter approach to define and interpret $(a,b)$-majorization. 

\sloppy Despite this discrepancy between ``classical" majorization and $(a,b)$-majorization, Lemmas~\ref{lemma_maj_sub}-\ref{lemma_maj_abs} can be partially generalized to the case of $(a,b)$-majorization. More precisely, to prove that a vector $x \in \mathbb{R}^n$ is $(a,b)$-majorized by another vector $y \in \mathbb{R}^n$ in one of the four senses, it suffices to check the majorization property for only a specific class of max-functions:
\begin{lemma}
Given $x,y \in \mathbb{R}^n$, the following hold:
\begin{enumerate}
\item
\sloppy $x \prec^{\downarrow}_{(a,b)} y$ if and only if $\sum_{i \in N} a_i \max (0, \alpha + \frac{x_i + b_i}{a_i}  ) \leq \sum_{i \in N} a_i \max (0, \alpha + \frac{y_i + b_i}{a_i}  )$ for all $\alpha \in \mathbb{R}$;
\item
$x \prec^{\uparrow}_{(a,b)} y$ if and only if $\sum_{i \in N} a_i \max (0, \alpha - \frac{x_i + b_i}{a_i}  ) \leq \sum_{i \in N} a_i \max (0, \alpha - \frac{y_i + b_i}{a_i}  )$ for all $\alpha \in \mathbb{R}$;
\item
$x \prec_{(a,b)} y$ if and only if $x \prec^{\downarrow}_{(a,b)} y$ and $x \prec^{\uparrow}_{(a,b)} y$;
\item
\sloppy $x \prec^{\text{abs}}_{(a,b)} y$ if and only if $\sum_{i \in N} a_i \max (\alpha - \frac{x_i + b_i}{a_i}  , 0 , \alpha + \frac{x_i + b_i}{a_i}  ) \leq \sum_{i \in N} a_i \max (\alpha - \frac{y_i + b_i}{a_i}  , 0 , \alpha + \frac{y_i + b_i}{a_i}  )$ for all $\alpha \in \mathbb{R}$.
\end{enumerate}
\label{lemma_leastab}
\end{lemma}
\begin{proof}
Proof of (1): The ``only if"-part follows directly from the definition of weak $(a,b)$-submajorization since $\max(0,\zeta)$ is non-decreasing    and convex    in $\zeta$. Thus, we proceed to prove the ``if"-part. Given a non-decreasing continuous convex function $\phi$, consider the piecewise linear approximation $\hat{\phi}$ of $\phi$ whose breakpoints are in the set $A(x,y) := \lbrace \frac{x_1 + b_1}{a_1} , \ldots, \frac{x_n + b_n}{a_n} , \frac{y_1 + b_1}{a_1} ,\ldots, \frac{y_n + b_n}{a_n}  \rbrace$    and that agrees with $\phi$ on these breakpoints, i.e., $\hat{\phi}(z) = \phi(z)$ for all $z \in A(x,y)$.    Moreover, let $\delta' := \min (A(x,y))$. Observe that for any $\zeta \in \mathbb{R}$, $\hat{\phi}(\zeta)$ can be written as $\phi(\delta') + \sum_{\delta \in A(x,y)} \beta_{\delta} \max(0,\alpha_{\delta} + \zeta)$ for some values $\alpha_{\delta} \in \mathbb{R}$    and $\beta_{\delta} \in \mathbb{R}_{> 0}$.    Thus, we have
\begin{align*}
\sum_{i \in N} a_i \phi \left( \frac{x_i + b_i}{a_i}  \right)
&= \sum_{i \in N} a_i \left( \phi(\delta') + \sum_{\delta \in A(x,y)}  \beta_{\delta}  \max \left(0,\alpha_{\delta} + \frac{x_i + b_i}{a_i}  \right) \right) \\
& \leq
 \sum_{i \in N} a_i \left( \phi(\delta') + \sum_{\delta \in A(x,y)}  \beta_{\delta}  \max \left(0,\alpha_{\delta} + \frac{y_i + b_i}{a_i}  \right) \right) \\
&= \sum_{i \in N} a_i \phi \left( \frac{y_i + b_i}{a_i}  \right).
\end{align*}
It follows that $x \prec^{\downarrow}_{(a,b)} y$.

Proof of (2): Analogous to the proof of (1).

Proof of (3): Analogous to the proof of (1), where we instead use the observation that the piecewise linear approximation of any continuous convex function $\phi$ with breakpoints in set $A(x,y)$ can be written as $\phi(\delta') + \sum_{\delta \in A; \delta > \delta'}  \beta_{\delta}  \max(0,\alpha_{\delta} - z) + \sum_{\delta \in A; \delta < \delta'}   \beta_{\delta}  \max(0,\alpha_{\delta} + z)$ for some values $\alpha_{\delta} \in \mathbb{R}$ and $\beta_{\delta} \in \mathbb{R}_{> 0}$,   
 where $\delta' := \min(A(x,y))$. Furthermore, letting $\alpha' := \max_{\delta \in A(x,y)} |\delta |$, the fact that $x \prec^{\downarrow}_{(a,b)} y$ implies that
\begin{align*}
a(N) \alpha' + x(N) +  b(N) &= \sum_{i \in N} a_i  \max \left( 0, \alpha' + \frac{x_i + b_i}{a_i}  \right) \\
& \leq \sum_{i \in N} a_i  \max \left( 0, \alpha' + \frac{y_i + b_i}{a_i} \right) \\
&= a(N) \alpha' + y(N) + b(N).
\end{align*}
It follows that $x(N) \leq y(N)$. Analogously, $x \prec^{\uparrow}_{(a,b)} y$ implies that $x(N) \geq y(N)$ and thus we have $x(N) = y(N)$.

Proof of (4): Analogous to the proof of (1), where we instead use the observation that the piecewise linear approximation of any even convex function $\phi$ with breakpoints in the set $A(x,y)$ can be written as $\phi(\delta') + \sum_{ \delta \in A(x,y)} \beta_{\delta} \max ( \alpha_{\delta} - z, 0 , \alpha_{\delta} + z)$ for some values $\alpha_{\delta} \in \mathbb{R}$ and    $\beta_{\delta} \in \mathbb{R}_{> 0}$,    where $\delta' := \min(A(x,y))$.
\end{proof}

   Finally, we    note that for any $b \in \mathbb{R}^n$ a least $(1,b)$-majorized element of some set $C \subseteq \mathbb{R}^n$ also minimizes $\Phi(x + b)$ over $C$ for any choice of Schur-convex function $\Phi$. Analogously, a least weakly $(1,b)$-sub- or $(1,b)$-supermajorized element or least weakly absolutely $(1,b)$-majorized element of $C$ minimizes $\Phi(x + b)$ for any choice of non-decreasing, non-increasing, or monotonically even Schur-convex function $\Phi$, respectively.

\section{Identifying sets with least $(a,b)$-majorized elements}
\label{sec_existence}

In this section, we establish the existence of least $(a,b)$-majorized, least weakly $(a,b)$-sub- and $(a,b)$-supermajorized, and least weakly absolutely $(a,b)$-majorized elements for several sets that are in some way related to (bi)submodular functions and the related polyhedra. The starting point for our investigations is a recently proved result in \cite{SchootUiterkamp2021} that establishes the existence of least $(a,b)$-majorized elements for base polyhedra and of least $(1,b)$-majorized elements for integral base polyhedra:
\begin{lemma}[Condition~1 and Theorem~1 in \cite{SchootUiterkamp2021}]
Let $f$ be a submodular function. For each pair of vectors $a \in \mathbb{R}^n_{>0}$ and $b \in \mathbb{R}^n$, the base polyhedron $B(f)$ has a unique least $(a,b)$-majorized element. Moreover, if $f$ and $b$ are integral, then any optimal solution to the problem $\min_{x \in B^{\mathbb{Z}}(f)} \sum_{i \in N} \frac{1}{2}(x_i + b_i)^2 $ is a least $(1,b)$-majorized element of the integral base polyhedron $B^{\mathbb{Z}}(f)$.
\label{lemma_reduction}
\end{lemma}

A natural question is whether base polyhedra are the only sets containing least $(a,b)$-majorized elements. The answer to this question is no:    any least $(a,b)$-majorized element of a given base polyhedron is also a least $(a,b)$-majorized element of any subset of the base polyhedron that contains this element.    Thus, the existence of least $(a,b)$-majorized elements is not limited to ``nice" sets such as polyhedra and convex sets. However, we do show in Section~\ref{sec_char} that base polyhedra are the only compact convex subsets of $\mathbb{R}^n$ that have least $(a,b)$-majorized elements for \emph{all} pairs $(a,b)$.

Lemma~\ref{lemma_reduction} forms the basis for all existence results in the remainder of this section. We present our existence results for least weakly $(a,b)$-sub- and $(a,b)$-supermajorized elements in Section~\ref{sec_exist_subsuper} and those for least weakly absolutely $(a,b)$-majorized elements in Section~\ref{sec_exist_bi}.

\subsection{Existence of least weakly $(a,b)$-submajorized and $(a,b)$-supermajorized elements}
\label{sec_exist_subsuper}

In this subsection, we focus on the existence of least weakly $(a,b)$-sub- and $(a,b)$-supermajorized elements. First, we use the initial existence result in Lemma~\ref{lemma_reduction} to prove that submodular and supermodular polyhedra have least weakly $(a,b)$-super- and $(a,b)$-submajorized elements, respectively:
\begin{lemma}
Let a submodular function $f$ and vectors $a \in \mathbb{R}^n_{>0}$ and $b \in \mathbb{R}^n$ be given. 
\begin{itemize}
\item
Any least $(a,b)$-majorized element of $B(f)$ is also a least weakly $(a,b)$-supermajorized element of $P(f)$ and a least weakly $(a,b)$-submajorized element of $P_{\text{sup}}(f^{\#})$;
\item
If $f$ and $b$ are integral, then any least $(1,b)$-majorized element of $B^{\mathbb{Z}}(f)$ is also a least weakly $(1,b)$-supermajorized element of $P^{\mathbb{Z}}(f)$ and a least weakly $(1,b)$-submajorized element of $P^{\mathbb{Z}}_{\text{sup}}(f^{\#})$.
\end{itemize}
\label{lemma_reduction_P}
\end{lemma}
\begin{proof}
We prove the first statement for the case of $P(f)$ (the proofs for the cases of $P_{\text{sup}}(f^{\#})$ and the second statement are analogous). Let $x^*$ be a least $(a,b)$-majorized element of $B(f)$, which exists due to Lemma~\ref{lemma_reduction}. Given a vector $z \in P(f)$, it follows from Lemma~\ref{lemma_ineq} that there exists a vector $y \in B(f)$ with $z \leq y$. Then we have for any non-increasing continuous convex function $\phi$ that
$
    \sum_{i \in N} a_i \phi (\frac{x_i^* + b_i}{a_i}  )
    \leq
        \sum_{i \in N} a_i \phi (\frac{y_i + b_i}{a_i}  )
    \leq     \sum_{i \in N} a_i \phi (\frac{z_i + b_i}{a_i}  ),
$
where the first inequality follows since $x^*$ is a least $(a,b)$-majorized element of $B(f)$ and the second inequality since $z \leq y$ and $\phi$ is non-increasing. Because $z \in P(f)$ was chosen arbitrarily, it follows that $x^*$ is a least weakly $(a,b)$-supermajorized element of $P(f)$.
\end{proof}

It follows directly from Lemma~\ref{lemma_reduction_P} that any set that contains a least $(a,b)$-majorized element of a base polyhedron and that is contained in the corresponding submodular or supermodular polyhedron, has a least weakly $(a,b)$-super- or $(a,b)$-submajorized element, respectively:
\begin{corollary}
Let a submodular function $f$ and vectors $a \in \mathbb{R}^n_{>0}$ and $b \in \mathbb{R}^n$ be given. Any least $(a,b)$-majorized element $x^*$ of $B(f)$ is also a 
\begin{itemize}
\item 
least weakly $(a,b)$-supermajorized element of any subset of $P(f)$ that contains $x^*$;
\item
least weakly $(a,b)$-submajorized element of any subset of $P_{\text{sup}}(f^{\#})$ that contains $x^*$.
\end{itemize}
Moreover, if $f$ and $b$ are integral, then any least $(1,b)$-majorized element $x^*$ of $B^{\mathbb{Z}}(f)$ is also a
\begin{itemize}
\item 
least weakly $(1,b)$-supermajorized element of any subset of $P^{\mathbb{Z}}(f)$ that contains $x^*$;
\item
least weakly $(1,b)$-submajorized element of any subset of $P^{\mathbb{Z}}_{\text{sup}}(f^{\#})$ that contains $x^*$.
\end{itemize}
\label{lemma_BCP}
\end{corollary}
Analogously to    Lemma~\ref{lemma_reduction},    Corollary~\ref{lemma_BCP} implies that sets containing least weakly $(a,b)$-sub- or $(a,b)$-supermajorized elements need not be ``nice".

Corollary~\ref{lemma_BCP} allows us to prove the existence of least weakly $(a,b)$-sub and -supermajorized elements for sets that are extensions or generalizations of submodular and supermodular polyhedra. In particular, this applies to bisubmodular polyhedra:
\begin{lemma}
Let a bisubmodular function $h$ and vectors $a \in \mathbb{R}^n_{>0}$ and $b \in \mathbb{R}^n$ be given. Then the following hold:
\begin{itemize}
\item
 Any least $(a,b)$-majorized element of the base polyhedron $B(h_{(N,\emptyset)})$ is also a least weakly $(a,b)$-supermajorized element of $\tilde{B}(h)$;
\item
 Any least $(a,b)$-majorized element of the base polyhedron $B_{\text{sup}}(-h_{(\emptyset,N)})$ is also a least weakly $(a,b)$-submajorized element of $\tilde{B}(h)$;
\end{itemize}
Moreover, if $h$ and $b$ are integral, then also the following hold:
\begin{itemize}
\item
 Any least $(1,b)$-majorized element of the integral base polyhedron $B^{\mathbb{Z}}(h_{(N,\emptyset)})$ is also a least weakly $(1,b)$-supermajorized element of $\tilde{B}^{\mathbb{Z}}(h)$;
\item
 Any least $(1,b)$-majorized element of the integral base polyhedron $B^{\mathbb{Z}}_{\text{sup}}(-h_{(\emptyset,N)})$ is also a least weakly $(1,b)$-submajorized element of $\tilde{B}^{\mathbb{Z}}(h)$;
\end{itemize}
\label{lemma_weakly_bi}
\end{lemma}
\begin{proof}
We prove that any least $(a,b)$-majorized element of $B(h_{(N,\emptyset)})$ is also a least weakly $(a,b)$-supermajorized element of $\tilde{B}(h)$ (the proofs of the other three statements are analogous).    Recall from Section~\ref{sec_bisubmodular} that, since $(N,\emptyset)$ is an orthant, we have $B_{(N,\emptyset)}(h) \subseteq \tilde{B}(h) \subseteq P_{(N,\emptyset)}(h)$. Also, note that for the specific orthant $(N,\emptyset)$, we have $P_{(N,\emptyset)}(h) = P(h_{(N,\emptyset)}) $ and $B_{(N,\emptyset)}(h) = B(h_{(N,\emptyset)})$. Thus, $\tilde{B}(h)$ is a subset of the submodular polyhedron $P(h_{(N,\emptyset)})$ and contains the corresponding base polyhedron $B(h_{(N,\emptyset}))$. Now the result of the lemma follows from Corollary~\ref{lemma_BCP}. 
\end{proof}

Alternatively, for integral bisubmodular polyhedra, Lemma~\ref{lemma_weakly_bi} can be proven by recognizing such polyhedra as special cases of jump systems \cite{Bouchet1995} and slightly adjusting the proof of Theorem~2.1 in \cite{Ando1996_jump}, which shows that jump systems have both a least weakly $(1,0)$-sub- and $(1,0)$-supermajorized element.

Lemma~\ref{lemma_weakly_bi} implies that also all special cases of bisubmodular polyhedra have both a least weakly $(a,b)$-sub- and $(a,b)$-supermajorized element. One example of such a special case are (bounded) generalized polymatroids (see, e.g., \cite{Bouchet1995}). In particular, Lemma~\ref{lemma_weakly_bi} generalizes Corollary~3.3 of \cite{Tamir1995}, where it is proven that bounded generalized polymatroids have both a least weakly $(1,0)$-sub- and $(1,0)$-supermajorized element.

\subsection{Existence of least weakly absolutely $(a,b)$-majorized elements}
\label{sec_exist_bi}

\todo{Fix todo removal error}

In this subsection, we prove the existence of least weakly absolutely $(a,b)$-majorized elements for submodular and supermodular polyhedra (Lemma~\ref{lemma_P_even}) and bisubmodular polyhedra (Lemma~\ref{lemma_even_bi}).

\begin{lemma}
Let a submodular function $f$, a supermodular function $g$, and vectors $a \in \mathbb{R}^n_{>0}$ and $b \in \mathbb{R}^n$ be given. Then the following hold:
\begin{itemize}
\item
Any least weakly $(a,b)$-supermajorized element of $P(f^{-b})$ is also a least weakly absolutely $(a,b)$-majorized element of $P(f)$;
\item
Any least weakly $(a,b)$-submajorized element of $P_{\text{sup}}(g^{-b})$ is also a least weakly absolutely $(a,b)$-majorized element of $P_{\text{sup}}(g)$.
\end{itemize}
Moreover, if $f$, $g$, and $b$ are integral, then also the following hold:
\begin{itemize}
\item
Any least weakly $(1,b)$-supermajorized element of $P^{\mathbb{Z}}(f^{-b})$ is also a least weakly absolutely $(1,b)$-majorized element of $P^{\mathbb{Z}}(f)$;
\item
Any least weakly $(1,b)$-submajorized element of $P^{\mathbb{Z}}_{\text{sup}}(g^{-b})$ is also a least weakly absolutely $(1,b)$-majorized element of $P^{\mathbb{Z}}_{\text{sup}}(g)$.
\end{itemize}
\label{lemma_P_even}
\end{lemma}
\begin{proof}
We prove that any least weakly $(a,b)$-supermajorized element of $P(f^{-b})$ is also a least weakly absolutely $(a,b)$-majorized element of $P(f)$ (the proofs of the other three statements are analogous). Consider the reduction $f^{-b}$ and let $x^*$ be a least weakly $(a,b)$-supermajorized element of $P(f^{-b})$ (recall that $f^{-b}$ is submodular and that $x^*$ exists by Lemma~\ref{lemma_reduction_P}). Note that, by definition of $P(f^{-b})$, we have $x^* \leq -b$ and thus $\frac{x^*_i + b_i}{a_i}  \leq  0$ for all $i \in N$.

Let a vector $x \in P(f)$ be given and define the vector $\tilde{x} \in \mathbb{R}^n$ as $\tilde{x}_i := \min (x_i, -b_i)$ for $i \in N$. Note that $\tilde{x} \in P(f)$ since $\tilde{x} \leq x$ and $x \in P(f)$. It follows that $\tilde{x} \in P(f^{-b})$ since $\tilde{x} \leq {-b}$. Moreover, we have $\frac{\tilde{x}_i + b_i}{a_i}  \leq 0$. Let an even continuous convex function $\phi$ be given and let
\begin{equation*}
\tilde{\phi}(\zeta) := \begin{cases}
\phi(\zeta) & \text{if } \zeta \leq 0; \\
\phi(0) & \text{if } \zeta \geq 0.
\end{cases}
\end{equation*}
We can now derive the following:
\begingroup
\begin{subequations}
\allowdisplaybreaks
\begin{align}
 \sum_{i \in N} a_i \phi \left(\frac{x^*_i + b_i }{a_i} \right)
&= \sum_{i \in N} a_i \tilde{\phi}\left(\frac{x^*_i + b_i }{a_i}  \right) \label{eq_P_even_01} \\
&\leq \sum_{i \in N} a_i \tilde{\phi}\left(\frac{\tilde{x}_i + b_i }{a_i}  \right) \label{eq_P_even_02}\\
&=  \sum_{i \in N} a_i \phi \left(\frac{\tilde{x}_i + b_i }{a_i}  \right) \label{eq_P_even_03}\\
&= \sum_{i: \tilde{x}_i = x_i} a_i \phi \left( \frac{x_i + b_i }{a_i}  \right) 
+ \sum_{i: \tilde{x}_i = -b_i < x_i} a_i \phi (0) \nonumber \\
& \leq \sum_{i \in N} a_i \phi \left( \frac{x_i + b_i }{a_i}  \right), \nonumber
\end{align}
\end{subequations}
\endgroup
where~(\ref{eq_P_even_01}) follows since $\frac{x_i^* + b_i }{a_i}  \leq 0$ for all $i \in N$, (\ref{eq_P_even_02}) since $x^*$ is a least weakly $(a,b)$-supermajorized element of $P(f^{-b})$, $\tilde{\phi}$ is non-increasing, and $\tilde{x} \in P(f^{-b})$, and (\ref{eq_P_even_03}) since $\frac{\tilde{x}_i + b_i }{a_i}  \leq 0$ for all $i \in N$. Since both $x$ and $\phi$ were chosen arbitrarily, it follows that $x^*$ is a least weakly absolutely $(a,b)$-majorized element of $P(f)$.
\end{proof}

Note that, contrary to least weakly $(a,b)$-submajorized and $(a,b)$-supermajorized elements, a least weakly absolutely $(a,b)$-majorized element of a submodular or supermodular polyhedron does not necessarily lie in the corresponding base polyhedron, even if the former does not contain the vector $-b$. An example of this is the submodular polyhedron    defined    by the function $f: 2^{\lbrace 1,2 \rbrace} \rightarrow \mathbb{R}$ with $f(\lbrace 1 \rbrace) = -1$, $f(\lbrace 2 \rbrace) = 2$, and $f(\lbrace 1,2 \rbrace) = 0$, whose least weakly absolutely $(1,0)$-majorized element is $(-1,0) \not\in B(f)$.

We are now ready to prove that also bisubmodular polyhedra have least weakly absolutely $(a,b)$-majorized elements:
\begin{lemma}
Let a bisubmodular function $h$ and vectors $a \in \mathbb{R}^n_{>0}$ and $b \in \mathbb{R}^n$ be given. Then the following hold:
\begin{itemize}
\item Any solution to the problem $\min_{x \in \tilde{B}(h)} \sum_{i\in N}  \frac{1}{2} \frac{(x_i + b_i)^2}{a_i} $ is a least weakly absolutely $(a,b)$-majorized element of $\tilde{B}(h)$;
\item If $h$ and $b$ are integral, then any solution to the problem $\min_{x \in \tilde{B}^{\mathbb{Z}}(h)} \sum_{i\in N} \frac{1}{2} (x_i + b_i)^2$ is a least weakly absolutely $(1,b)$-majorized element of $\tilde{B}^{\mathbb{Z}}(h)$.
\end{itemize}
\label{lemma_even_bi}
\end{lemma}
\begin{proof} 
We prove the lemma for the case of bisubmodular polyhedra (the proof for the case of integral bisubmodular polyhedra is analogous). Let $h$ be a bisubmodular function. If $-b \in \tilde{B}(h)$, then $-b$ is the unique least weakly absolutely $(a,b)$-majorized element of $\tilde{B}(h)$ since for any $x \in \mathbb{R}^n$ and even continuous convex function $\phi$ we have $
\sum_{i \in N} a_i \phi ( \frac{x_i + b_i}{a_i} )
\geq \sum_{i \in N} a_i \phi(0)
= \sum_{i \in N} a_i \phi (\frac{-b_i + b_i}{a_i}  ).
$
Thus, suppose to the contrary that $-b \not\in \tilde{B}(h)$ and let $x^*$ be an optimal solution to the problem $\min_{x \in \tilde{B}(h)}  \sum_{i \in N} \frac{1}{2} \frac{(x_i + b_i)^2}{a_i} $. We define the following partition of the index set $N$ based on the difference between $x^*$ and $-b$:
\begin{align*}
A_+ &:= \lbrace i \in N \ | \ x_i^* < -b_i \rbrace; &
A_- &:= \lbrace i \in N \ | \ x_i^* > -b_i \rbrace; &
A_0 &:= \lbrace i \in N \ | \ x_i^* = -b_i \rbrace.
\end{align*}
   It can be shown by means of an exchange argument that there exists a pair    $(\tilde{X},\tilde{Y}) \in 3^N$    such that    $(\tilde{X},\tilde{Y}) \sqsupseteq (A_+,A_-)$    and    $x^*(\tilde{X}) - x^*(\tilde{Y}) = h(\tilde{X},\tilde{Y})$ (see, e.g., Theorem~3.1 in \cite{Fujishige1997}).        Let $(S,T)$ be an orthant that contains $(\tilde{X},\tilde{Y})$. Finally, we introduce the following notation: for a given vector $v \in \mathbb{R}^n$, let $\tilde{v}$ be a corresponding vector with $\tilde{v}_i = - v_i$ if $i \in \tilde{Y}$ and $\tilde{v}_i = v_i$ otherwise.

The proof consists of two parts, namely showing that $\tilde{x}^* \in B((h_{(S,T)})^{-\tilde{b}})$ and that  $\tilde{x} \in B((h_{(S,T)})^{-\tilde{b}})$ implies $x \in \tilde{B}(h)$. This implies that $\tilde{x}^*$ is a least $(a,\tilde{b})$-majorized element of $B((h_{(S,T)})^{-\tilde{b}})$ since $x^*$ minimizes $\sum_{i \in N} \frac{1}{2} \frac{(x_i + b_i)^2}{a_i}$ over $\tilde{B}(h)$ and $B((h_{(S,T)})^{-\tilde{b}})$ is a base polyhedron containing $x^*$. It follows from Lemma~\ref{lemma_reduction_P} that $\tilde{x}^*$ is a least weakly $(a,\tilde{b})$-submajorized element of $P((h_{(S,T)})^{-\tilde{b}})$ and subsequently from Lemma~\ref{lemma_P_even} that $\tilde{x}^*$ is a least weakly absolutely $(a,\tilde{b})$-majorized element of $P(h_{(S,T)})$. Thus, $x^*$ is a least weakly absolutely $(a,b)$-majorized element of $P_{(S,T)}(h)$ and also of $\tilde{B}(h)$ since $\tilde{B}(h) \subseteq P_{(S,T)}(h)$.

For the first part, note that $\tilde{x}^* \leq -\tilde{b}^*$. Also, we have $\tilde{x}^* \in P(h_{(S,T)})$. Together, this means that $\tilde{x}^* \in P((h_{(S,T)})^{-\tilde{b}})$. Furthermore, we have that
\begin{align*}
\tilde{x}^*(N) &
\leq h^{-\tilde{b}}_{(S,T)} (N)
= \min (h_{(S,T)} (Z) - \tilde{b}(N \backslash Z) \ | \ Z \subseteq N) \\
& \leq h_{(S,T)} (\tilde{X} \cup \tilde{Y}) - \tilde{b}(N \backslash (\tilde{X} \cup \tilde{Y}))
= h(\tilde{X},\tilde{Y}) - b(N \backslash (\tilde{X} \cup \tilde{Y})) \\
&= x^*(\tilde{X}) - x^*(\tilde{Y}) + x^*(N \backslash (\tilde{X} \cup \tilde{Y}))
= \tilde{x}^*(N).
\end{align*}
Thus, $\tilde{x}^*(N) = h^{-\tilde{b}}_{(S,T)} (N)$ and we have $\tilde{x}^* \in B((h_{(S,T)})^{-\tilde{b}})$.

For the second part, note that for all $\tilde{x} \in  B((h_{(S,T)})^{-\tilde{b}})$ we have $\tilde{x} \in P(h_{(S,T)})$, $\tilde{x}(\tilde{X} \cup \tilde{Y}) = h_{(S,T)}(\tilde{X} \cup \tilde{Y})$ and $\tilde{x}_i = -\tilde{b}_i$ for all $i \not\in \tilde{X} \cup \tilde{Y}$. Thus, $x \in P_{(S,T)}(h)$, $x(\tilde{X}) - x(\tilde{Y}) = h(\tilde{X},\tilde{Y})$ and $x_i = -b_i$ for all $i \not\in \tilde{X} \cup \tilde{Y}$. We show that these three properties of $x$ imply $x \in \tilde{B}(h)$\footnote{Note that if $(\tilde{X},\tilde{Y})$ is an orthant (meaning $(S,T) = (\tilde{X},\tilde{Y})$), then $x \in B_{(S,T)}(h)$. Thus, we already know that $x \in \tilde{B}(h)$ since $B_{(S,T)}(h) \subseteq \tilde{B}(h)$. However, it is not guaranteed that $(\tilde{X},\tilde{Y})$ is an orthant. Hence an extended proof is necessary.}.    For this, we choose an arbitrary $(X,Y) \in 3^N$ and define $X_1 := X \cap (\tilde{X} \cup \tilde{Y})$, $X_2 := X \backslash (\tilde{X} \cup \tilde{Y})$, $Y_1 := Y \cap (\tilde{X} \cup \tilde{Y})$, and $Y_2 := Y \backslash (\tilde{X} \cup \tilde{Y})$ (note that all sets $X_1,X_2,Y_1,Y_2$ are disjoint). We prove that $x(X) - x(Y) \leq h(X,Y)$. First, if $(X_1,Y_1) \sqsubseteq (\tilde{X},\tilde{Y})$, then
\begin{subequations}
\label{eq_bi_even_01}
\begingroup \allowdisplaybreaks
\begin{align}
x(X) - x(Y) &= x(X_1) - x(Y_1) - b(X_2) + b(Y_2) \label{eq_bi_even_01_01}\\
&= x(X_1) - x(Y_1) + x^*(\tilde{X} \cup X_2) - x^*(\tilde{X})
- x^*(\tilde{Y} \cup Y_2) + x^*(\tilde{Y}) \label{eq_bi_even_01_02} \\
& \leq h(X_1,Y_1) + h(\tilde{X} \cup X_2,\tilde{Y} \cup Y_2) - h(\tilde{X},\tilde{Y}) \label{eq_bi_even_01_03} \\
&\leq h(X,Y), \label{eq_bi_even_01_04}
\end{align}
\endgroup
\end{subequations}
where~(\ref{eq_bi_even_01_01}) follows since $X_2$, $Y_2$, and $\tilde{X} \cup \tilde{Y}$ are disjoint and $x_i = -b_i$ for all $i \not\in \tilde{X} \cup \tilde{Y}$, (\ref{eq_bi_even_01_02}) since additionally    $(A_+ , A_-) \sqsubseteq (\tilde{X},\tilde{Y})$    and thus $X_2,Y_2 \subseteq A_0$, (\ref{eq_bi_even_01_03}) since $x \in P_{(\tilde{S} , \tilde{T})}(h)$,    $(X_1,Y_1) \sqsubseteq (\tilde{X},\tilde{Y}) (h)$,    $x^* \in \tilde{B}(h)$, and $x(\tilde{X}) - x(\tilde{Y})  = h(\tilde{X},\tilde{Y})$, and~(\ref{eq_bi_even_01_04}) by bisubmodularity of $h$. Second, if $(X_1,Y_1) \not\sqsubseteq (\tilde{X},\tilde{Y})$, we have that
\begin{align}
(X \cup \tilde{X}) \backslash (Y \cup \tilde{Y}) 
= (\tilde{X} \backslash Y_1) \cup X_2 \label{eq_bi_even_help_01}
\end{align}
and, analogously, $(Y \cup \tilde{Y}) \backslash (X \cup \tilde{X}) = (\tilde{Y} \backslash X_1) \cup Y_2$ (see Figure~\ref{fig_lemma_even_bi}). Moreover, since $((\tilde{X} \backslash Y_1) \cup X_2 ) \cap (\tilde{X} \cup \tilde{Y}) = \tilde{X} \backslash Y \subseteq \tilde{X}$ and $((\tilde{Y} \backslash X_1) \cup Y_2) \cap (\tilde{X} \cup \tilde{Y}) = \tilde{Y} \backslash X  \subseteq \tilde{Y}$, it follows from~(\ref{eq_bi_even_01}) that 
\begin{align}
& \quad \ x((X \cup \tilde{X}) \backslash (Y \cup \tilde{Y})) - x((Y \cup \tilde{Y}) \backslash (X \cup \tilde{X})) \nonumber \\
&\leq h((X \cup \tilde{X}) \backslash (Y \cup \tilde{Y}) ,(Y \cup \tilde{Y}) \backslash (X \cup \tilde{X}))  \label{eq_bi_even_help_02}
\end{align}
and also $x(X \cap \tilde{X}) - x(Y \cap \tilde{Y}) \leq h(X \cap \tilde{X},Y \cap \tilde{Y})$. We are now ready to derive the following:
\begin{subequations}
\begin{align}
& \quad \ x(X) - x(Y) \\
& = x(X_1 \cap \tilde{X}) + x(X_1 \cap \tilde{Y}) - x(Y_1 \cap \tilde{X}) - x(Y_1 \cap \tilde{Y}) + x(X_2) - x(Y_2) \label{eq_bi_even_02_02} \\
& = x(X \cap \tilde{X}) + x(X_1 \cap \tilde{Y}) - x(Y_1 \cap \tilde{X}) - x(Y \cap \tilde{Y}) + x(X_2) - x(Y_2) \label{eq_bi_even_02_03} \\
&= x(X \cap \tilde{X})  - x(Y \cap \tilde{Y}) - h(\tilde{X},\tilde{Y}) \nonumber \\
& \quad + x(\tilde{X} \backslash Y_1) - x(\tilde{Y} \backslash X_1)  + x(X_2) - x(Y_2) \label{eq_bi_even_02_04}  \\
&= x(X \cap \tilde{X})  - x(Y \cap \tilde{Y}) - h(\tilde{X},\tilde{Y}) \nonumber \\
& \quad + x((X \cup \tilde{X}) \backslash (Y \cup \tilde{Y})) - x((Y \cup \tilde{Y}) \backslash (X \cup \tilde{X})) \label{eq_bi_even_02_05}  \\
&\leq h(X \cap \tilde{x}, Y \cap \tilde{Y}) - h(\tilde{X},\tilde{Y}) + h((X \cup \tilde{X}) \backslash (Y \cup \tilde{Y}),(Y \cup \tilde{Y}) \backslash (X \cup \tilde{X})) \label{eq_bi_even_02_06}  \\
& \leq h(X,Y), \label{eq_bi_even_02_07}
\end{align}
\end{subequations}
where~(\ref{eq_bi_even_02_02}) follows since $X_1,X_2 \subseteq \tilde{X} \cup \tilde{Y}$, (\ref{eq_bi_even_02_03}) since $X_2 \cap \tilde{X} = Y_2 \cap \tilde{Y} = \emptyset$, (\ref{eq_bi_even_02_04}) since $x(\tilde{X}) - x(\tilde{Y}) = h(\tilde{X},\tilde{Y})$, (\ref{eq_bi_even_02_05}) from (\ref{eq_bi_even_help_01}), (\ref{eq_bi_even_02_06}) from (\ref{eq_bi_even_help_02}), and (\ref{eq_bi_even_02_07}) by bisubmodularity of $h$. Thus, we have $x(X) - x(Y) \leq h(X,Y)$ and, since $(X,Y)$ was chosen arbitrarily from $3^N$, it follows that $x \in \tilde{B}(h)$.
\end{proof}

\begin{figure}[t!]
\centering
\subfloat[]{\includegraphics{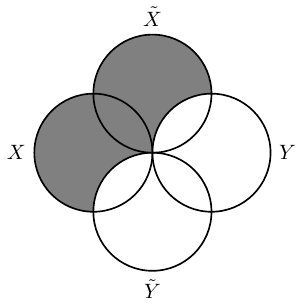}}
\hspace{.5cm}
\subfloat[]{\includegraphics{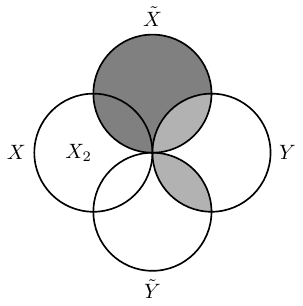}}
\caption{Venn-diagrams of $X$, $Y$, $\tilde{X}$, and $\tilde{Y}$ in the    first    part of the proof of Lemma~\ref{lemma_even_bi}: (a) $(X \cup \tilde{X}) \backslash (Y \cup \tilde{Y})$; (b) $Y_1$ (light gray), $\tilde{X} \backslash Y_1$ (dark gray), and $X_2$.}
\label{fig_lemma_even_bi}
\end{figure}

\section{From $(a,b)$-majorization to (bi)sumbodularity}
\label{sec_from_maj_to_sub}

In the previous section, we showed the existence of least majorized elements for several sets including submodular, supermodular, base, and bisubmodular polyhedra. In this section, we show that the existence of such elements is limited to these sets within the class of compact convex or bounded hole-free sets. Together with the existence results of the previous section, this yields characterizations of the existence of least majorized elements in terms of submodular, supermodular, base, and bisubmodular polyhedra. We characterize the existence of least weakly $(a,b)$-super- and $(a,b)$-submajorized, least $(a,b)$-majorized, and least weakly absolutely $(a,b)$-majorized elements in Sections~\ref{sec_char_super}-\ref{sec_char_bi}, respectively.

\subsection{Least weakly $(a,b)$-supermajorized elements and submodular polyhedra}
\label{sec_char_super}

To characterize the existence of least weakly $(a,b)$-supermajorized elements, we first prove two intermediate results. First, in Lemma~\ref{lemma_red_to_sub_1}, we prove that for any closed and bounded subset $C \subset \mathbb{R}^n$ the existence of a least $(1,b)$-majorized element for all $b \in \mathbb{R}^n$ implies the existence of vectors whose nested sums are maximal among vectors in $C$:
\begin{lemma}
Let $C \subset \mathbb{R}^n$ be a closed and bounded set. Suppose that $C$ has a least weakly $(1,b)$-supermajorized element for each $b \in \mathbb{R}^n$. Then for each permutation $\pi$ of $N$, there exists a vector $x^* \in C$ such that for all $x \in C$ we have $\sum_{i=1}^k x_{\pi(i)}^* \geq \sum_{i=1}^k x_{\pi(i)}$ for all $k \in N$.
\label{lemma_red_to_sub_1}
\end{lemma}
\begin{proof}
For each $i \in N$, we define $l_i$ and $u_i$ as the minimum and maximum value of a variable $x_i$ in $C$, i.e., $l_i := \min_{x \in C} x_i$ and $u_i := \max_{x \in C} x_i$. Given a permutation $\pi$ of $N$, we recursively define the following vector:
\begin{align*}
\hat{b}_{\pi(1)} &:= l_{\pi(1)}; &
\hat{b}_{\pi(i+1)} &:= \hat{b}_{\pi(i)} + u_{\pi(i)} - l_{\pi(i+1)} + 1, \quad i \in \lbrace 1 ,\ldots,n - 1 \rbrace.
\end{align*}
Moreover, we define for each $k \in N$ the convex function $\hat{\phi}^k(\zeta) := \max(0,\hat{b}_{\pi(k)} + u_{\pi(k)} - \zeta)$. Note that $\hat{b}_{\pi(i+1)} + u_{\pi(i+1)} \geq \hat{b}_{\pi(i)} + u_{\pi(i)} + 1$ for all    $i \in \lbrace 1,\ldots, n-1 \rbrace$    and thus that, for each $k \in N$,
\begin{equation*}
\sum_{i \in N} \hat{\phi}^k(x_i + \hat{b}_i)
=
-\sum_{i=1}^k (\hat{b}_{\pi(i)} + x_{\pi(i)}) + k(\hat{b}_{\pi(k)} + u_{\pi(k)}).
\end{equation*}
It follows from Lemma~\ref{lemma_leastab} that any least weakly $(1,\hat{b})$-supermajorized element $x^* \in C$ minimizes $\sum_{i \in N} \hat{\phi}^k(x_i + \hat{b}_i)$ over $C$ for all $k \in N$. Thus, for all $x \in C$ and $k \in N$ we have
\begin{align*}
-\sum_{i=1}^k (x_{\pi(i)}^*+\hat{b}_{\pi(i)}) + k(\hat{b}_{\pi(k)} + u_{\pi(k)})
&=
\sum_{i \in N} \hat{\phi}^k(x_i^* + \hat{b}_i)
\leq
\sum_{i \in N} \hat{\phi}^k(x_i + \hat{b}_i) \\
&=
-\sum_{i=1}^k ( x_{\pi(i)} +\hat{b}_{\pi(i)})  + k(\hat{b}_{\pi(k)} + u_{\pi(k)}),
\end{align*}
which implies that $\sum_{i=1}^k x_{\pi(i)}^* \geq \sum_{i=1}^k x_{\pi(i)}$.
\end{proof}

\begin{remark}
\label{remark_diff}
Note, that the functions $\hat{\phi}^k$ as defined in the proof of this lemma are not (continuously) differentiable. Alternatively, we can prove the result of the lemma starting from the assumption that for any $b \in \mathbb{R}^n$ there exists a vector in $C$ that is an optimal solution to $\min_{x \in C} \sum_{i \in N} \phi (x_i + b_i)$ for any \emph{continuously differentiable} convex function $\phi$. The corresponding proof is equal to that of Lemma~\ref{lemma_red_to_sub_1}, except that for each $k \in N$ we choose $\hat{\phi}^k$ as the following continuously differentiable and convex function:
\begin{equation*}
\hat{\phi}^k(\zeta) = \begin{cases}
\hat{b}_{\pi(k)} - u_{\pi(k)} - \zeta + \frac{1}{2} & \text{if } \zeta \leq \hat{b}_{\pi(k)} - u_{\pi(k)}; \\
\begin{array}{l} -\frac{1}{2} (\hat{b}_{\pi(k)} - u_{\pi(k)} + 1 - \zeta)^4 \\ \quad + (\hat{b}_{\pi(k)} - u_{\pi(k)} + 1 - \zeta)^3 \end{array} & \begin{array}{l} \text{if } \hat{b}_{\pi(k)} - u_{\pi(k)} \leq \zeta \\ \quad \leq \tilde{b}_{\pi(k)} - u_{\pi(k)} + 1; \end{array} \\
0 & \text{if } \zeta \geq \tilde{b}_{\pi(k)} - u_{\pi(k)} + 1.
\end{cases}
\end{equation*}
We come back to this when proving our characterization results for least $(a,b)$-majorized elements in Theorems~\ref{th_base} and~\ref{th_base_int}.
\end{remark}

In a second step, we prove in Lemma~\ref{lemma_red_to_sub_2} that the convex hull of any closed and bounded set satisfying the result of Lemma~\ref{lemma_red_to_sub_1} is contained    in    a submodular polyhedron and contains the corresponding base polyhedron. The proof of this lemma is inspired by the proof of Theorem~1 in \cite{Nakamura1988}, where the optimality of Edmonds' classical greedy algorithm for linear optimization \cite{Edmonds2003} is characterized in terms of submodular polyhedra.
\begin{lemma}
Let $C \subset \mathbb{R}^n$ be a closed and bounded set. If for each permutation $\pi$ of $N$ there exists a vector $x^* \in C$ such that for all $x \in C$ we have $\sum_{i=1}^k x_{\pi(i)}^* \geq \sum_{i=1}^k x_{\pi(i)}$ for all $k \in N$, then the set function $f(A) := \max_{x \in C} \ x(A)$ is submodular and we have $B(f) \subseteq \text{co}(C) \subseteq P(f)$.
\label{lemma_red_to_sub_2}
\end{lemma}
\begin{proof}
We first prove that the set function $f$ as defined in the lemma is submodular. Given two sets $A,B \in 2^N$, we define $X := A \cap B$, $Y := A \backslash B$, and $Z := B \backslash A$. Let $\pi$ be a permutation whose first $|X|$ elements are indices in $X$    and whose next $|Y \cup Z|$ elements are indices in $Y \cup Z$.    By assumption, there exists a solution $x^*$ such that $x^*(X) \geq x(X)$ and $x^*(X \cup Y \cup Z) \geq x(X \cup Y \cup Z)$ for all $x \in C$. It follows that
\begin{align*}
f(A) + f(B)  &=
f(X \cup Y) + f(X \cup Z) =
\max_{x \in C} x(X \cup Y)
+
\max_{x \in C} x(X \cup Z)  \\
& \geq
x^*(X \cup Y) + x^*(X \cup Z) 
= x^*(X) + x^*(X \cup Y \cup Z) \\
&= \max_{x \in C} x(X)
+
\max_{x \in C} x(X \cup Y \cup Z) 
= f(X) + f(X \cup Y \cup Z) \\
&= f(A \cap B) + f(A \cup B),
\end{align*}
which implies that $f$ is submodular.

Finally, we prove that $B(f) \subseteq \text{co}(C) \subseteq P(f)$. To prove that $\text{co}(C) \subseteq P(f)$, note that each vector $y \in \text{co}(C)$ can be written as a convex combination of vectors in $C$, i.e., $y = \sum_{j=1}^m \lambda^j x^j$ for some positive values $\lambda^1,\ldots, \lambda^m \in \mathbb{R}_{>0}$ with $\sum_{j=1}^m \lambda^j = 1$ and vectors $x^j \in C$. It follows that
\begin{equation*}
    y(A) = \sum_{j=1}^m \lambda^j x^j(A)
    \leq \max_{x \in \lbrace x^1,\ldots,x^m \rbrace} x(A)
    \leq \max_{x \in C} x(A) = f(A), \quad A \in 2^N.
\end{equation*}
Thus, $\text{co}(C) \subseteq P(f)$. To prove that $\text{co}(C) \supseteq B(f)$, suppose that there exists a vector $z \in B(f)$ that is not in $\text{co}(C)$. Since $\text{co}(C)$ is compact and convex, we may assume without loss of generality that $z$ is an extreme point of $B(f)$. It follows from Lemma~\ref{lemma_extreme} that there exists a permutation $\pi$ of $N$ such that $z_{\pi(1)} = f(\lbrace \pi(1) \rbrace)$ and, for each $k \geq 1$, we have $
z_{\pi(k+1)} =
f(\lbrace \pi(1),\ldots,\pi(k+1) \rbrace) 
- f(\lbrace \pi(1),\ldots,\pi(k) \rbrace)$. By assumption, there exists a vector $y^* \in C$ such that for all $x \in C$ we have $\sum_{i=1}^k y^*_{\pi(i)} \geq \sum_{i=1}^k x_{\pi(i)}$ for all $k \in N$. It follows that
\begin{equation*}
z_{\pi(1)} = f(\lbrace \pi(1) \rbrace) = \max_{x \in C} x_{\pi(1)} = y_{\pi(1)}^*
\end{equation*}
and for each $k \geq 1$, we have
\begin{align*}
z_{\pi(k+1)} &=
f(\lbrace \pi(1),\ldots,\pi(k+1) \rbrace) 
- f(\lbrace \pi(1),\ldots,\pi(k) \rbrace)  \\
&= \max_{x \in C} \ \sum_{i=1}^{k+1} x_{\pi(i)}
- \max_{x \in C} \ \sum_{i=1}^{k} x_{\pi(i)} 
= \sum_{i=1}^{k+1} y_{\pi(i)}^* -  \sum_{i=1}^{k} y_{\pi(i)}^* 
= y_{\pi(k+1)}^*.
\end{align*}
This implies that $z = y^*$ and thus that $z \in \text{co}(C)$. This is a contradiction, which means that we have $\text{co}(C) \supseteq B(f)$.
\end{proof}

When the set $C$ in Lemma~\ref{lemma_red_to_sub_2} is convex, it follows that $C$ itself is contained in a submodular polyhedron and contains the corresponding base polyhedron since $\text{co}(C) = C$. Furthermore,    we show in Corollary~\ref{cor_red_to_sub_Z} that    when $C$ is an integral hole-free set, we can adjust the proof of the lemma slightly so that we may conclude that $C$ is contained in an integral submodular polyhedron and contains the corresponding integral base polyhedron:
\begin{corollary}
Let $C \subseteq \mathbb{Z}^n$ be a bounded integral hole-free set. If for each permutation $\pi$ of $N$ there exists a vector $x^* \in C$ such that for all $x \in C$ we have $\sum_{i=1}^k x_{\pi(i)}^* \geq \sum_{i=1}^k x_{\pi(i)}$ for all $k \in N$, then the set function $f(A) := \max_{x \in C} \ x(A)$ is submodular and we have $B^{\mathbb{Z}}(f) \subseteq C \subseteq P^{\mathbb{Z}}(f)$.
\label{cor_red_to_sub_Z}
\end{corollary}
\begin{proof}
The result follows from two slight adjustments of the proof of Lemma~\ref{lemma_red_to_sub_2}, where we now aim to prove that $B^{\mathbb{Z}}(f) \subseteq C \subseteq P^{\mathbb{Z}}(f)$. First, we have $C \subseteq P^{\mathbb{Z}}(f)$ since for each $y \in C$ and $A \in 2^N$ it holds that $y(A) \leq \max_{x \in C} x(A) = f(A)$. Seond, to prove that $C \supseteq B^{\mathbb{Z}}(f)$, we now suppose that there exists a vector $z \in B^{\mathbb{Z}}(f)$ that is not in $C$. Since $C$ is bounded and hole-free, we may assume without loss of generality that $z$ is an extreme point of $B(f)$. Following the remainder of the proof, we may conclude that $z \in C$.
\end{proof}

We are now ready to prove our two characterizations of the existence of least weakly $(a,b)$-supermajorized elements:
\begin{theorem}
Let $C \subset \mathbb{R}^n$ be a compact convex set. The following statements are equivalent:
\begin{enumerate}
    \item $C$ has a least weakly $(a,b)$-supermajorized element for each $a \in \mathbb{R}^n_{>0}$ and $b \in \mathbb{R}^n$;
    \item $C$ has a least weakly $(1,b)$-supermajorized element for each $b \in \mathbb{R}^n$;
\item For each $b \in \mathbb{R}^n$, there exists $x^* \in C$ that is an optimal solution to $\min_{x \in C} \Phi(x + b)$ for any choice of non-increasing    continuous    Schur-convex function $\Phi$; 
\item For each permutation $\pi$ of $N$, there exists $x^* \in C$ such that for all $x \in C$ we have $\sum_{i=1}^k x^*_{\pi(i)} \geq \sum_{i=1}^k x_{\pi(i)}$ for all $k \in N$;
    \item The function $f(A) := \max_{x \in C} x(A)$ is submodular and $B(f) \subseteq C \subseteq P(f)$;
\end{enumerate}
\label{th_super}
\end{theorem}
\begin{proof}
(2) is a special case of (1); (2) and (3) are equivalent due to Lemma~\ref{lemma_maj_sup}; (2) implies (4) via Lemma~\ref{lemma_red_to_sub_1}; (4) implies (5) via Lemma~\ref{lemma_red_to_sub_2} since $\text{co}(C) = C$; (5) implies (1) via    Lemma~\ref{lemma_reduction} and    Corollary~\ref{lemma_BCP}.
\end{proof}

\begin{theorem}
Let $C \subset \mathbb{Z}^n$ be a bounded integral hole-free set. The following statements are equivalent:
\begin{enumerate}
    \item $C$ has a least weakly $(1,b)$-supermajorized element for each $b \in \mathbb{Z}^n$;
\item For each $b \in \mathbb{Z}^n$, there exists $x^* \in C$ that is an optimal solution to $\min_{x \in C} \Phi(x + b)$ for any choice of non-increasing    continuous    Schur-convex function $\Phi$;
\item For each permutation $\pi$ of $N$, there exists $x^* \in C$ such that for all $x \in C$ we have $\sum_{i=1}^k x^*_{\pi(i)} \geq \sum_{i=1}^k x_{\pi(i)}$ for all $k \in N$;
    \item The function $f(A) := \max_{x \in C} x(A)$ is integral and submodular and $B^{\mathbb{Z}}(f) \subseteq C \subseteq P^{\mathbb{Z}}(f)$.
\end{enumerate}
\label{th_super_int}
\end{theorem}
\begin{proof}
(1) and (2) are equivalent due to Lemma~\ref{lemma_maj_sup}; (1) implies (3) via Lemma~\ref{lemma_red_to_sub_1}; (3) implies (4) via Corollary~\ref{cor_red_to_sub_Z}; (4) implies (1) via Corollary~\ref{lemma_BCP}.
\end{proof}

\subsection{Least weakly $(a,b)$-submajorized elements and supermodular polyhedra}

A characterization analogous to Theorems~\ref{th_super} and~\ref{th_super_int} can be proven for the existence of least weakly $(a,b)$-submajorized elements. This proof relies on the following intermediate results, whose proofs are analogous to those of Lemmas~\ref{lemma_red_to_sub_1} and~\ref{lemma_red_to_sub_2} and Corollary~\ref{cor_red_to_sub_Z}, respectively:
\begin{lemma}
Let $C \subset \mathbb{R}^n$ be a closed and bounded set. Suppose that $C$ has a least weakly $(1,b)$-submajorized element for each $b \in \mathbb{R}^n$. Then for each permutation $\pi$ of $N$, there exists a vector $x^* \in C$ such that for all $x \in C$ we have $\sum_{i=1}^k x_{\pi(i)}^* \leq \sum_{i=1}^k x_{\pi(i)}$ for all $k \in N$.
\label{lemma_red_to_sub_1_02}
\end{lemma}
\begin{lemma}
Let $C \subset \mathbb{R}^n$ be a closed and bounded set. If for each permutation $\pi$ of $N$ there exists a vector $x^* \in C$ such that for all $x \in C$ we have $\sum_{i=1}^k x_{\pi(i)}^* \leq \sum_{i=1}^k x_{\pi(i)}$ for all $k \in N$, then the set function $g(A) := \min_{x \in C} \ x(A)$ is supermodular and we have $B_{\text{sup}}(g) \subseteq \text{co}(C) \subseteq P_{\text{sup}}(g)$.
\label{lemma_red_to_sub_2_02}
\end{lemma}
\begin{corollary}
Let $C \subseteq \mathbb{Z}^n$ be a bounded integral hole-free set. If for each permutation $\pi$ of $N$ there exists a vector $x^* \in C$ such that for all $x \in C$ we have $\sum_{i=1}^k x_{\pi(i)}^* \leq \sum_{i=1}^k x_{\pi(i)}$ for all $k \in N$, then the set function $g(A) := \min_{x \in C} \ x(A)$ is supermodular and we have $B_{\text{sup}}^{\mathbb{Z}}(g) \subseteq C \subseteq P_{\text{sup}}^{\mathbb{Z}}(g)$.
\label{cor_red_to_sub_Z_02}
\end{corollary}
These results lead to the following characterization of least weakly $(a,b)$-submajorized elements, whose proofs are analogous to those of Theorems~\ref{th_super} and~\ref{th_super_int}, respectively:
\begin{theorem}
Let $C \subset \mathbb{R}^n$ be a compact convex set. The following statements are equivalent:
\begin{enumerate}
    \item $C$ has a least weakly $(a,b)$-submajorized element for each $a \in \mathbb{R}^n_{>0}$ and $b \in \mathbb{R}^n$;
    \item $C$ has a least weakly $(1,b)$-submajorized element for each $b \in \mathbb{R}^n$;
\item For each $b \in \mathbb{R}^n$, there exists $x^* \in C$ that is an optimal solution to $\min_{x \in C} \Phi(x + b)$ for any choice of non-decreasing    continuous    Schur-convex function $\Phi$;
\item For each permutation $\pi$ of $N$, there exists $x^* \in C$ such that for all $x \in C$ we have $\sum_{i=1}^k x^*_{\pi(i)} \leq \sum_{i=1}^k x_{\pi(i)}$ for all $k \in N$;
    \item The function $g(A) := \min_{x \in C} x(A)$ is supermodular and $B_{\text{sup}}(g) \subseteq C \subseteq P_{\text{sup}}(g)$;
\end{enumerate}
\label{th_sub}
\end{theorem}

\begin{theorem}
Let $C \subset \mathbb{Z}^n$ be a bounded integral hole-free set. The following statements are equivalent:
\begin{enumerate}
    \item $C$ has a least weakly $(1,b)$-submajorized element for each $b \in \mathbb{Z}^n$;
\item For each $b \in \mathbb{Z}^n$, there exists $x^* \in C$ that is an optimal solution to $\min_{x \in C} \Phi(x + b)$ for any choice of non-decreasing    continuous    Schur-convex function $\Phi$;
\item For each permutation $\pi$ of $N$, there exists $x^* \in C$ such that for all $x \in C$ we have $\sum_{i=1}^k x^*_{\pi(i)} \leq \sum_{i=1}^k x_{\pi(i)}$ for all $k \in N$;
    \item The function $g(A) := \min_{x \in C} x(A)$ is integral and supermodular and $B_{\text{sup}}^{\mathbb{Z}}(g) \subseteq C \subseteq P_{\text{sup}}^{\mathbb{Z}}(g)$.
\end{enumerate}
\label{th_sub_int}
\end{theorem}

\subsection{Least $(a,b)$-majorized elements and base polyhedra}
\label{sec_char}

Using the characterization of the existence of least weakly $(a,b)$-submajorized elements in Theorems~\ref{th_super} and~\ref{th_super_int}, we obtain the following characterization of least $(a,b)$-majorized elements in terms of (integral) base polyhedra:
\begin{theorem}
Let $C \subset \mathbb{R}^n$ be a compact convex set. The following statements are equivalent:
\begin{enumerate}
    \item $C$ has a least $(a,b)$-majorized element for each $a \in \mathbb{R}^n_{>0}$ and $b \in \mathbb{R}^n$;
    \item $C$ has a least $(1,b)$-majorized element for each $b \in \mathbb{R}^n$;
\item For each $b \in \mathbb{R}^n$, there exists $x^* \in C$ that is an optimal solution to $\min_{x \in C} \Phi(x + b)$ for any choice of    continuous    Schur-convex function $\Phi$;
    \item The function $f(A) := max_{x \in C} x(A)$ is submodular and $C = B(f)$.
\end{enumerate}
\label{th_base}
\end{theorem}
\begin{proof}
(2) is a special case of (1); (2) and (3) are equivalent due to Lemma~\ref{lemma_maj}; (4) implies (1) by Lemma~\ref{lemma_reduction}. To prove that (2) implies (4), note that (2) implies that all elements in $C$ have the same element-sum. Moreover, it follows from Theorem~\ref{th_super} that $B(f) \subseteq C \subseteq P(f)$ for the submodular function $f(A) := \max_{x \in C} x(A)$. Since $x(N) = f(N)$ for all $x \in B(f)$, it follows that $x(N) = f(N)$ for all $x \in C$ and thus $C \subseteq B(f)$. We may therefore conclude that $C = B(f)$.
\end{proof}

\begin{theorem}
Let $C \subset \mathbb{Z}^n$ be a bounded integral hole-free set. The following statements are equivalent:
\begin{enumerate}
    \item $C$ has a least $(1,b)$-majorized element for each $b \in \mathbb{Z}^n$;
\item For each $b \in \mathbb{Z}^n$, there exists $x^* \in C$ that is an optimal solution to $\min_{x \in C} \Phi(x + b)$ for any choice of    continuous    Schur-convex function $\Phi$;
    \item The function $f(A) := \max_{x \in C} x(A)$ is integral and submodular and $C = B^{\mathbb{Z}}(f)$.
\end{enumerate}
\label{th_base_int}
\end{theorem}
\begin{proof}
The proof is analogous to that of Theorem~\ref{th_base}.
\end{proof}

We conclude this subsection with a result that may be of independent interest. Theorems~\ref{th_base} and~\ref{th_base_int} imply that a particular optimality condition for minimizing separable convex functions over (integral) base polyhedra is unique for these problems. For a given feasible set $C$ and a vector $x \in C$, this condition is as follows (see also, e.g., Theorem~8.1 in  \cite{Fujishige2005}):
\begin{condition}
Let a continuous convex function $\phi_i$ be given for each $i \in N$. For each pair $(i,k) \in N^2$ such that $x_i + \alpha (e^k - e^i) \in C$ for some $\alpha >0$, we have $\phi_k^+(x_k) \geq \phi_i^-(x_i)$, where $\phi_k^+$ is the right derivative of $\phi_k$ and $\phi_i^-$ is the left derivative of $\phi_i$.
\label{cond_base}
\end{condition}
The proof of the existence of least $(a,b)$-majorized elements in (integral) base polyhedra (see also Lemma~\ref{lemma_reduction}) is directly based on this optimality condition. This implies the following characterization of the validity of this condition:
\begin{corollary}
Let $C_1 \subset \mathbb{R}^n$ be a compact convex set and $C_2 \subset \mathbb{Z}^n$ be a bounded integral hole-free set. Then the following hold:
\begin{itemize}
\item
Condition~\ref{cond_base} is a valid optimality condition for the problem $\min_{x \in C_1} \sum_{i \in N} \phi_i(x_i)$ for any choice of continuous convex functions $\phi_i$, $i \in N$, if and only if $C_1$ is a base polyhedron;
\item
Condition~\ref{cond_base} is a valid optimality condition for the problem $\min_{x \in C_2} \sum_{i \in N} \phi_i(x_i)$ for any choice of continuous convex functions $\phi_i$,    $i \in N$,    if and only if $C_2$ is an integral base polyhedron.
\end{itemize}
\end{corollary}
\begin{proof}
The ``if"-parts follow from, e.g., Theorem~8.1 in \cite{Fujishige2005}.   
Regarding the ``only if'' parts, Condition~\ref{cond_base} implies via Theorem~1 in \cite{SchootUiterkamp2021} the existence of least $(a,b)$-majorized elements for each pair of vectors $a \in \mathbb{R}^n_{>0}$ and $b \in \mathbb{R}^n$ (see also Lemma~\ref{lemma_reduction}). It follows from parts (1) and (4) of Theorem~\ref{th_base} that $C_1$ is a base polyhedron and from parts (1) and (3) of Theorem~\ref{th_base_int} that $C_2$ is an integral base polyhedron.

\end{proof}

\subsection{Least weakly absolutely $(a,b)$-majorized elements and bisubmodular polyhedra}
\label{sec_char_bi}

In this section, we prove that bisubmodular polyhedra are the only compact convex sets with least weakly absolutely $(a,b)$-majorized elements for all $a \in \mathbb{R}^n_{>0}$ and $b \in \mathbb{R}^n$ and that integral bisubmodular polyhedra are the only bounded integral hole-free sets with least weakly absolutely $(1,b)$-majorized elements for all $b \in \mathbb{Z}^n$. For this, we follow the same proof idea as in Section~\ref{sec_char_super} for the case of least weakly $(a,b)$-supermajorized elements and submodular polyhedra. First, we prove in Lemma~\ref{lemma_red_to_bi_1} that, given a closed and bounded set $C \subset \mathbb{R}^n$, the existence of least weakly absolutely $(1,b)$-elements for all $b\in \mathbb{R}^n$ implies for each permutation $\pi$ and orthant $(S,T)$ the existence of a vector in $C$ that maximizes functions of the form $x(X) - x(Y)$, where $(X,Y) \sqsubseteq (S,T)$:

\begin{lemma}
Let $C \subset \mathbb{R}^n$ be a closed and bounded set. Suppose that $C$ has a least weakly absolutely $(1,b)$-majorized element for each $b \in \mathbb{R}^n$. Then for each permutation $\pi$ of $N$ and sign vector $s \in \lbrace -1,1 \rbrace^n$ there exists $x^* \in C$ such that for all $x \in C$ we have $\sum_{i=1}^k s_{\pi(i)} x^*_{\pi(i)} \geq \sum_{i=1}^k s_{\pi(i)} x_{\pi(i)}$ for all $k \in N$.
\label{lemma_red_to_bi_1}
\end{lemma}
\begin{proof}
Let a permutation $\pi$ and a sign vector $s$ be given and let $(S,T) \in 3^N$ be the unique orthant corresponding to $s$ i.e., $(S,T) := N \ | \ s$, i.e., $S := \lbrace i \in N \ | \ s_i = 1 \rbrace$ and $T:= \lbrace i \in N \ | \ s_i = -1 \rbrace$ (see also Section~\ref{sec_bisubmodular}). Moreover, let $l_i := \min_{x \in C} x_i$ and $u_i := \min_{x \in C} x_i$ for $i \in N$. Based on $(S,T)$, we define alternative bound vectors $\hat{l},\hat{u} \in \mathbb{R}^n$ as follows:
\begin{align*}
\hat{l}_i &:= \begin{cases}
u_i & \text{if } i \in S; \\
l_i & \text{if } i \in T;
\end{cases} &
\hat{u}_i &:=  \begin{cases}
l_i & \text{if } i \in S; \\
u_i & \text{if } i \in T.
\end{cases}
\end{align*}
Note that $s_i \hat{l}_i \geq s_i \hat{u}_i$ for all $i \in N$. We recursively define the vector $\hat{b}$ as follows:
\begin{align*}
\hat{b}_{\pi(n)} &:= - \hat{l}_{\pi(n)} - s_{\pi(n)}; \\
\hat{b}_{\pi(i)} &:= 
- s_{\pi(i)} ( | \hat{b}_{\pi(i+1)} + \hat{u}_{\pi(i+1)} | + 1) - \hat{l}_{\pi(i)}, \quad i \in \lbrace 1,\ldots, n-1 \rbrace.
\end{align*}

Note that for all $i$ with $\pi(i) \in S$ we have $0 \geq \hat{b}_{\pi(i)} + \hat{l}_{\pi(i)} \geq \hat{b}_{\pi(i)} + \hat{u}_{\pi(i)}$ and for all $i$ with $\pi(i) \in T$ we have $0 \leq \hat{b}_{\pi(i)} + \hat{l}_{\pi(i)} \leq \hat{b}_{\pi(i)} + \hat{u}_{\pi(i)}$.  It follows that, for any $x \in C$, we have for $i$ with $\pi(i) \in S$ that
\begin{equation}
-|\hat{b}_{\pi(i)} + \hat{l}_{\pi(i)}|
= \hat{b}_{\pi(i)} + \hat{l}_{\pi(i)}
\geq
\hat{b}_{\pi(i)} + x_{\pi(i)}
\geq \hat{b}_{\pi(i)} + \hat{u}_{\pi(i)}
= -|\hat{b}_{\pi(i)} + \hat{u}_{\pi(i)}|
\label{eq_tilde_S}
\end{equation}
and for $i$ with $\pi(i) \in T$ that
\begin{equation}
|\hat{b}_{\pi(i)} + \hat{l}_{\pi(i)}|
= \hat{b}_{\pi(i)} + \hat{l}_{\pi(i)}
\leq
\hat{b}_{\pi(i)} + x_{\pi(i)}
\leq \hat{b}_{\pi(i)} + \hat{u}_{\pi(i)}
= |\hat{b}_{\pi(i)} + \hat{u}_{\pi(i)}|.
\label{eq_tilde_T}
\end{equation}
Also, note that for all $i < n$ we have
\begin{equation}
| \hat{b}_{\pi(i)} + \hat{l}_{\pi(i)} |
= |\hat{b}_{\pi(i+1)} + \hat{u}_{\pi(i+1)} | + 1
> |\hat{b}_{\pi(i+1)} + \hat{l}_{\pi(i+1)} |
\label{eq_tilde_3}
\end{equation}
and for all $i > 1$ we have
\begin{equation}
|\hat{b}_{\pi(i)} + \hat{u}_{\pi(i)}| 
=
- s_{\pi(i-1)}(\hat{b}_{\pi(i-1)} + \hat{l}_{\pi(i-1)}) - 1
< |\hat{b}_{\pi(i-1)} + \hat{l}_{\pi(i-1)}|.
\label{eq_tilde_1}
\end{equation}

We define for each $k \in N$ the function $\hat{\phi}^k(y) := \max(-|\hat{b}_{\pi(k)} + \hat{l}_{\pi(k)}| - y, 0 , -|\hat{b}_{\pi(k)} + \hat{l}_{\pi(k)}| +y)$. For each $k \in N$, we have for $i$ with $\pi(i) \in S$ that
\begin{align*}
& \quad \
\hat{\phi}^k(x_{\pi(i)} + \hat{b}_{\pi(i)}) \\
&=
\max(-|\hat{b}_{\pi(k)} + \hat{l}_{\pi(k)}| - x_{\pi(i)} - \hat{b}_{\pi(i)}, 0 , -|\hat{b}_{\pi(k)} + \hat{l}_{\pi(k)}| + x_{\pi(i)} + \hat{b}_{\pi(i)}) \\
&=
\max( -|\hat{b}_{\pi(k)} + \hat{l}_{\pi(k)}| - x_{\pi(i)} - \hat{b}_{\pi(i)}, 0 ) \\
&= \begin{cases}
-|\hat{b}_{\pi(k)} + \hat{l}_{\pi(k)}| - x_{\pi(i)} - \hat{b}_{\pi(i)} & \text{if } i \leq k; \\
0 & \text{if } i > k,
\end{cases}
\end{align*}
where the second equality follows since $x_{\pi(i)} + \hat{b}_{\pi(i)} \leq 0$ by (\ref{eq_tilde_S}) and the third equality follows from (\ref{eq_tilde_S}) and~(\ref{eq_tilde_3}) (for $i \leq k$) and (\ref{eq_tilde_S}), (\ref{eq_tilde_1}), and (\ref{eq_tilde_3}) (for $i > k$). Analogously, we have for $i$ with $\pi(i) \in T$ that
\begin{align*}
& \quad \
\hat{\phi}^k(x_{\pi(i)} + \hat{b}_{\pi(i)}) \\
&=
\max(-|\hat{b}_{\pi(k)} + \hat{l}_{\pi(k)}| - x_{\pi(i)} - \hat{b}_{\pi(i)}, 0 , -|\hat{b}_{\pi(k)} + \hat{l}_{\pi(k)}| + x_{\pi(i)} + \hat{b}_{\pi(i)}) \\
&=
\max( 0 , -|\hat{b}_{\pi(k)} + \hat{l}_{\pi(k)}| + x_{\pi(i)} + \hat{b}_{\pi(i)}) \\
&= \begin{cases}
-|\hat{b}_{\pi(k)} + \hat{l}_{\pi(k)}| + x_{\pi(i)} + \hat{b}_{\pi(i)} & \text{if } i \leq k; \\
0 & \text{if } i > k,
\end{cases}
\end{align*}
where the second equality follows since $x_{\pi(i)} + \hat{b}_{\pi(i)} \geq 0$ by Equation~(\ref{eq_tilde_T}) and the third equality follows from (\ref{eq_tilde_T}) and~(\ref{eq_tilde_3}) (for $i \leq k$) and (\ref{eq_tilde_T}), (\ref{eq_tilde_1}), and (\ref{eq_tilde_3}) (for $i > k$). It follows that
\begin{equation*}
\sum_{i \in N} \hat{\phi}^k(x_i + \hat{b}_i)
=
-\sum_{i=1}^k  s_{\pi(i)} (x_{\pi(i)} + \hat{b}_{\pi(i)}) - k|\hat{b}_{\pi(k)} + \hat{l}_{\pi(k)}|.
\end{equation*}
By assumption, there exists a vector $x^* \in C$ that simultaneously minimizes these functions over $C$ for all $k \in N$. This means that for any $x \in C$ we have
\begin{align*}
& \quad \
-\sum_{i=1}^k   s_{\pi(i)} (x^*_{\pi(i)} + \hat{b}_{\pi(i)}) - k|\hat{b}_{\pi(k)} + \hat{l}_{\pi(k)}| \\
&= \sum_{i \in N} \hat{\phi}^k(x^*_i + \hat{b}_i) 
 \leq
\sum_{i \in N} \hat{\phi}^k(x_i + \hat{b}_i)\\
&=
-\sum_{i=1}^k  s_{\pi(i)} (x_{\pi(i)} + \hat{b}_{\pi(i)}) - k|\hat{b}_{\pi(k)} + \hat{l}_{\pi(k)}|,
\end{align*}
which implies that $\sum_{i=1}^k s_{\pi(i)} x^*_{\pi(i)} \geq \sum_{i=1}^k s_{\pi(i)} x_{\pi(i)}$ for all $k \in N$.
\end{proof}
\begin{remark}
\label{remark_diff_2}
Note that, analogously to Lemma~\ref{lemma_red_to_sub_1} for least $(1,b)$-majorized elements, we can prove the result of Lemma~\ref{lemma_red_to_bi_1} starting from the assumption that for each $b \in \mathbb{R}^n$ there exists a vector in $C$ that is an optimal solution to $\min_{x \in C} \sum_{i \in N} \phi(x_i + b_i)$ for any \emph{continuously differentiable} convex function $\phi$ (see also Remark~\ref{remark_diff}).
\end{remark}

In a second step, we prove in Lemma~\ref{lemma_red_to_bi_2} that the convex hull of any closed and bounded set satisfying the result of Lemma~\ref{lemma_red_to_bi_1} is a bisubmodular polyhedron. Analogously to the proof of Lemma~\ref{lemma_red_to_sub_2}, the proof of this lemma is inspired by the proof of Theorem~2 in \cite{Nakamura1988}, where the optimality of the greedy algorithm in \cite{Dunstan1973} for linear optimization is characterized in terms of bisubmodular polyhedra.
\begin{lemma}
Let $C \subset \mathbb{R}^n$ be a closed and bounded set. If for each permutation $\pi$ of $N$ and sign vector $s \in \lbrace -1,1 \rbrace^n$ there exists $x^* \in C$ such that for all $x \in C$ we have $\sum_{i=1}^k s_{\pi(i)} x^*_{\pi(i)} \geq \sum_{i=1}^k s_{\pi(i)} x_{\pi(i)}$ for all $k \in N$, then the convex hull of $C$ is a bisubmodular polyhedron    defined    by the biset function $h(X,Y) := \max_{x \in C} (x(X) - x(Y))$.
\label{lemma_red_to_bi_2}
\end{lemma}
\begin{proof}
First, we prove that $h$ is bisubmodular. For this, choose any two pairs $(S_1,T_1),(S_2,T_2) \in 3^N$. Since $S_1 \cap T_1 = S_2 \cap T_2 = \emptyset$, we have that $i \in S_1 \cap S_2$ implies $i \not\in T_1 \cup T_2$ and thus that $S_1 \cap S_2 \subseteq (S_1 \cup S_2) \backslash (T_1 \cup T_2)$. Analogously, we have $T_1 \cap T_2 \subseteq (T_1 \cup T_2) \backslash (S_1 \cup S_2)$. This means that there exists a permutation $\pi$ of $N$ that satisfies the following properties:
\begin{itemize}
\item
$\pi(i) \in 
(S_1 \cap S_2) \cup (T_1 \cap T_2)$ if $1 \leq i \leq |S_1 \cap S_2| + |T_1 \cap T_2|$;
\item
$\pi(i) \in
(S_1 \cup S_2) \backslash (T_1 \cup T_2) \cup (T_1 \cup T_2) \backslash (S_1 \cup S_2)$ if  $1 \leq i \leq |(S_1 \cup S_2) \backslash (T_1 \cup T_2) \cup (T_1 \cup T_2) \backslash (S_1 \cup S_2)|$.
\end{itemize}
By assumption, there exists a vector $x^* \in C$ such that for all $k\in N$ we have $\sum_{i=1}^k s_{\pi(i)} x^*_{\pi(i)} \geq \sum_{i=1}^k s_{\pi(i)} x_{\pi(i)}$ for all $x \in C$. In particular, this holds for $k = |S_1 \cap S_2| + |T_1 \cap T_2|$ and    $k = |(S_1 \cup S_2) \backslash (T_1 \cup T_2)| + |(T_1 \cup
T_2) \backslash (S_1 \cup S_2)|$,    which implies that $x^*(S_1 \cap S_2) - x^*(T_1 \cap T_2) \geq x(S_1 \cap S_2) - x(T_1 \cap T_2)$ and $x^*((S_1 \cup S_2) \backslash (T_1 \cup T_2)) - x^*( (T_1 \cup T_2) \backslash (S_1 \cup S_2)) \geq x((S_1 \cup S_2) \backslash (T_1 \cup T_2)) - x( (T_1 \cup T_2) \backslash (S_1 \cup S_2))$ for all $x \in C$. It follows that
\begingroup
\allowdisplaybreaks
\begin{align*}
& \quad \ h(S_1,T_1) + h(S_2,T_2) \\
&= \max_{x \in C} (x(S_1) - x(T_1))
+ \max_{x \in C} (x(S_2) - x(T_2)) \\
&\geq x^*(S_1) - x^*(T_1) +  x^*(S_2) - x^*(T_2) \\
&= x^*(S_1 \cup S_2) + x^*(S_1 \cap S_2) - x^*(T_1 \cup T_2) - x^*(T_1 \cap T_2) \\
&= x^*(S_1 \cup S_2) - x^*(S_1 \cup S_2 \cup T_1 \cup T_2) + x^*(S_1 \cup S_2 \cup T_1 \cup T_2) - x^*(T_1 \cup T_2) \\
& \quad + \max_{x \in C}(x(S_1 \cap S_2) - x(T_1 \cap T_2)) \\
&= -x^*((T_1 \cup T_2) \backslash (S_1 \cup S_2)) + x^*((S_1 \cup S_2) \backslash (T_1 \cup T_2)) + h(S_1 \cap S_2,T_1 \cap T_2) \\
& =\max_{x \in C}(x((S_1 \cup S_2) \backslash (T_1 \cup T_2)) - x((T_1 \cup T_2) \backslash (S_1 \cup S_2))) \\
& \quad + h(S_1 \cap S_2,T_1 \cap T_2) \\
&= h((S_1 \cup S_2) \backslash (T_1 \cup T_2),(T_1 \cup T_2) \backslash (S_1 \cup S_2)) + h(S_1 \cap S_2,T_1 \cap T_2).
\end{align*} \endgroup
Thus, $h$ is bisubmodular.

Finally, we prove that $\text{co}(C) = \tilde{B}(h)$. For this, first note that each $y \in \text{co}(C)$ can be written as a convex combination of vectors in $C$, i.e., $y = \sum_{j=1}^m \lambda^j x^j$ for some positive values $\lambda^1,\ldots, \lambda^m \in \mathbb{R}_{>0}$ with $\sum_{j=1}^m \lambda^j = 1$ and vectors $x^j \in C$. It follows that
\begin{align*}
    y(S) - y(T) &= \sum_{j=1}^m \lambda^j (x^j(S) - x^j(T))
    \leq \max_{x \in \lbrace x^1,\ldots,x^m \rbrace} (x(S) - x(T)) \\
&    \leq \max_{x \in C} (x(S) - x(T))= h(S,T)
\end{align*}
for all $(S,T) \in 3^N$. Thus, $\text{co}(C) \subseteq \tilde{B}(h)$. To prove that $\text{co}(C) \supseteq \tilde{B}(g)$, suppose that there exists a vector $z \in \tilde{B}(h)$ that is not in $\text{co}(C)$. Since $\text{co}(C)$ is convex, we may assume without loss of generality that $z$ is an extreme point of $\tilde{B}(h)$. It follows from Lemma~\ref{lemma_extreme_bi} that there exists a permutation $\pi$ of $N$ and a sign vector $s \in \lbrace -1,1 \rbrace^n$ such that $z_{\pi(1)} =  s_{\pi(1)} h(\lbrace \pi(1) \rbrace \ | \ s)$ and, for each $k > 1$, we have
\begin{equation*}
z_{\pi(k)} =  s_{\pi(k)} (h(\lbrace \pi(1),\ldots,\pi(k) \rbrace \ | \ s) - h(\lbrace \pi(1),\ldots,\pi(k-1) \rbrace \ | \ s) ).
\end{equation*}
By assumption, there exists a vector $y^* \in C$ such that for all $x \in C$ we have $\sum_{i = 1}^k s_{\pi(i)} y_{\pi(i)}^* \geq \sum_{i=1}^k s_{\pi(i)} x_{\pi(i)}$ for all $k \in N$. It follows that
\begin{equation*}
z_{\pi(1)} =  s_{\pi(1)} h(\lbrace \pi(1) \rbrace \ | \ s)
= s_{\pi(1)} \max_{x\in C} x(\lbrace \pi(1) \rbrace \ | \ s)
= y^*_{\pi(1)},
\end{equation*}
and
\begin{align*}
z_{\pi(k)} &=  s_{\pi(k)} (h(\lbrace \pi(1),\ldots,\pi(k) \rbrace \ | \ s) - h(\lbrace \pi(1),\ldots,\pi(k-1) \rbrace \ | \ s) \\
&=s_{\pi(k)} \max_{x \in C} (x(\lbrace \pi(1),\ldots,\pi(k) \rbrace^+) - x(\lbrace \pi(1),\ldots,\pi(k) \rbrace^-)) \\
& \quad
- s_{\pi(k)} \max_{x \in C} (x(\lbrace \pi(1),\ldots,\pi(k-1) \rbrace^+) - x(\lbrace \pi(1),\ldots,\pi(k-1) \rbrace^-)) \\
&= s_{\pi(k)} \max_{x \in C} \sum_{i=1}^k s_{\pi(i)} x_{\pi(i)} - s_{\pi(k)} \max_{x \in C} \sum_{i=1}^{k-1} s_{\pi(i)} x_{\pi(i)} \\
&= s_{\pi(k)}  \sum_{i=1}^k s_{\pi(i)} y^*_{\pi(i)} - s_{\pi(k)} \sum_{i=1}^{k-1} s_{\pi(i)} y^*_{\pi(i)} \\
&= s_{\pi(k)}^2 y^*_{\pi(k)} 
= y^*_{\pi(k)}.
\end{align*}
This implies that $z = y^*$ and thus that $z \in \text{co}(C)$. This is a contradiction, which means that we have $\text{co}(C) \supseteq \tilde{B}(h)$ and thus $\text{co}(C) = \tilde{B}(h)$.
\end{proof}

Analogously to the case of submodular and base polyhedra in Lemma~\ref{lemma_red_to_sub_2} and Corollary~\ref{cor_red_to_sub_Z},   
it follows that any \emph{convex} set $C$ satisfying the requirements of Lemma~\ref{lemma_red_to_bi_2} is necessarily a bisubmodular polyhedron since it equals its convex hull, and that any \emph{integral hole-free} set $C$ satisfying the requirements of the lemma is an integral bisubmodular polyhedron since it equals the integral points of its convex hull.    As a consequence, we can now prove our two main characterization results for the existence of least weakly absolutely $(a,b)$-majorized elements and (integral) bisubmodular polyhedra:
\begin{theorem}
Let $C \subset \mathbb{R}^n$ be a compact convex set. The following statements are equivalent:
\begin{enumerate}
    \item $C$ has a least weakly absolutely $(a,b)$-majorized element for each $a \in \mathbb{R}^n_{>0}$ and $b \in \mathbb{R}^n$;
    \item $C$ has a least weakly absolutely $(1,b)$-majorized element for each $b \in \mathbb{R}^n$;
\item
For each $b \in \mathbb{R}^n$, there exists $x^* \in C$ that is an optimal solution to $\min_{x \in C} \Phi(x+b)$ for any choice of    monotonically    even    continuous    Schur-convex function $\Phi$;
\item  For each permutation $\pi$ of $N$ and sign vector $s \in \lbrace -1,1 \rbrace^n$ there exists $x^* \in C$ such that for all $x \in C$ we have $\sum_{i=1}^k s_{\pi(i)} x^*_{\pi(i)} \geq \sum_{i=1}^k s_{\pi(i)} x_{\pi(i)}$ for all $k \in N$;
\item The function $h(X,Y) := \max_{x \in C} (x(X) - x(Y))$ is bisubmodular and $C = \tilde{B}(h)$.
\end{enumerate}
\label{th_bi}
\end{theorem}
\begin{proof}
(2) is a special case of (1); (2) and (3) are equivalent due to Lemma~\ref{lemma_maj_abs}; (2) implies (4) via Lemma~\ref{lemma_red_to_bi_1}; (4) implies (5) via Lemma~\ref{lemma_red_to_bi_2}; (5) implies (1) via Lemma~\ref{lemma_even_bi}.
\end{proof}

\begin{theorem}
Let $C \subset \mathbb{Z}^n$ be a bounded integral hole-free set. The following statements are equivalent:
\begin{enumerate}
    \item $C$ has a least weakly absolutely $(1,b)$-majorized element for each $b \in \mathbb{Z}^n$;
\item
For each $b \in \mathbb{Z}^n$, there exists $x^* \in C$ that is an optimal solution to $\min_{x \in C} \Phi(x+b)$ for any choice of    monotonically    even    continuous    Schur-convex function $\Phi$;
\item  For each permutation $\pi$ of $N$ and sign vector $s \in \lbrace -1,1 \rbrace^n$ there exists $x^* \in C$ such that for all $x \in C$ we have $\sum_{i=1}^k s_{\pi(i)} x^*_{\pi(i)} \geq \sum_{i=1}^k s_{\pi(i)} x_{\pi(i)}$ for all $k \in N$;
\item The function $h(X,Y) := \max_{x \in C} (x(X) - x(Y))$ is integral and bisubmodular and $C = \tilde{B}^{\mathbb{Z}}(h)$.
\end{enumerate}
\label{th_bi_int}
\end{theorem}
\begin{proof}
The proof is analogous to that of Theorem~\ref{th_bi}.
\end{proof}

We conclude this section by noting that in all Theorems~\ref{th_super}-\ref{th_bi_int}, the definition of least majorized elements can be relaxed to \emph{continuously differentiable} convex functions due to Remarks~\ref{remark_diff} and~\ref{remark_diff_2}.

\section{Applications}
\label{sec_appl}

In this section, we demonstrate the impact of our characterization results have in several fields other than combinatorial optimization. In particular, we highlight the insights that the results provide in three application areas, namely power management, cooperative game theory, and regularized regression.

\subsection{Energy storage scheduling and power management}

To illustrate the impact of our characterization results on RAPs in power management applications, we take as example the energy storage scheduling problem described in \cite[Section~5.2]{SchootUiterkamp2021}. Given a time horizon consisting of $n$ equidistant time intervals of length $\Delta t$ indexed by the set $N := \lbrace 1,\ldots, n \rbrace$ we determine for each interval $i \in N$ the (dis)charging power $x_i$ of the storage system during this interval. The goal is to optimize a given system objective function of the form $\Phi (x + p)$, where $p \in \mathbb{R}^n$ denotes the uncontrollable energy usage of the system and its environment (e.g., a household or neighborhood). Given the initial and target amounts of energy $S_{\text{start}}$ and $S_{\text{end}}$ at the start and end of the time horizon, respectively, the scheduling problem is formulated as follows:
\begin{align*}
\min_{x \in \mathbb{R}^n} \ &  \Phi (x + p) \\
\text{s.t. } & 0 \leq S_{\text{start}} + \Delta t \cdot  x(\lbrace 1, \ldots, j \rbrace) \leq D, \quad j \in N  \backslash \lbrace n \rbrace, \\
& S_{\text{start}} + \Delta t \cdot x(N) = S_{\text{end}}, \\
& X_{\text{min}} \leq x_i \leq X_{\text{max}}, \quad i \in N.
\end{align*}
The feasible region of this problem is a base polyhedron and thus Lemma~\ref{lemma_reduction} (Condition~1 and Theorem~1 in \cite{SchootUiterkamp2021}) applies. Using this observation, it was established that the following objectives can be optimized simultaneously:
\begin{itemize}
\item Minimizing exchange with the main grid: $\Phi(x + p) = \sum_{i \in N} |x_i + p_i|$;
\item Load profile flattening: $\Phi(x + p) = \sum_{i \in N} (x_i + p_i)^2$;
\item Peak shaving under a threshold $M$: $\Phi(x + p) = \sum_{i \in N} \max(0,  \underline{f}(x_i + p_i))$ where $\underline{f}$ is a convex non-decreasing function with $\underline{f}(M) = 0$.
\end{itemize}

Now, utilizing our characterization result in Theorem~\ref{th_base}, we expand upon this result in two directions. First, we may extend the above collection of equivalent objectives with general norms. In particular, to model the widely used objective of (general) peak-shaving, i.e., without a threshold $M$, we may consider the max-norm $\Phi(x + p) = \max_{i \in N} |x_i + p_i|$. We may thus conclude that a solution that optimizes the (quadratic) objective of load profile flattening also optimizes the general peak-shaving objective.

Second, however, the characterization also suggests that these properties might not hold anymore for extensions and variants of the basic energy storage model that destroy the submodular structure. Examples of these are the inclusion of conversion losses \cite{Guo2013} or restricting the charging rate to a finite set of base rates \cite{vanderKlauw2017}. For these cases, there exist Schur-convex choices of $\Phi$ for which an optimizer of the load profile flattening objective is not optimal. An interesting question for future work is to investigate whether this includes the aforementioned choices for $\Phi$ that are relevant in energy storage scheduling.

Finally, we note that the observations in this subsection also apply to
other device scheduling problems such as electric vehicles and heat pumps \cite{vanderKlauw2017}, as well as to other power management problems with submodular structure, including power allocation in multichannel communication systems, vessel speed optimization, and speed scaling (see also \cite[Section~5]{SchootUiterkamp2021}).

\subsection{Convex cooperative games with transferable utility}

In cooperative game theory, one important subclass of games consists of those with transferable utility (TU). These games are defined by a player set $N$ (the grand coalition) and a set function $v$ that assigns a value to each coalition. The core of the game consists of all payoff allocations where no coalition has an incentive to split from the grand coalition $N$, meaning that each coalition $A \in 2^N$ receives at least their value $v(A)$ and the total payoff equals the value $v(N)$ of the grand coalition. A TU cooperative game is called convex if its value function is supermodular, in which case its core is a base polyhedron. Several characterizations of convexity of a game exist, for example in terms of extreme points of the core \cite{Ichiishi1981}. Here, we obtain a new characterization of convex games in terms of the existence of so-called egalitarian solutions \cite{Dutta1989}. Egalitarian solutions aim to distribute the payoff of the grand coalition as equally as possible over the players. By definition, such solutions are exactly the least majorized elements of the considered allocation space (see, e.g., \cite{Arin2003}). When considering egalitarian solutions that are restricted to the core, Theorem~\ref{th_base} directly gives us the following characterization of convex games:

\begin{theorem}
Let $(N,v)$ be a TU cooperative game with $n$ players. Then $(N,v)$ is a convex game if and only if for each $b \in \mathbb{R}^n$ the game $(N,v^b)$ has an egalitarian core solution, where $v^b(A) := v(A) + b(A)$ for all $A \in 2^N$.
\label{col_game}
\end{theorem}
\begin{proof}
This result is a reformulation of the equivalence between parts (2) and (4) of Theorem~\ref{th_base}.
\end{proof}

One interesting question is whether the characterization result in Theorem~\ref{th_bi} for bisubmodular polyhedra can be used to characterize bicooperative games \cite{Bilbao2008}, analogously to Theorem~\ref{col_game}. A first step for this, which we leave for future work, would be to find a suitable definition of egalitarian solutions for this type of game and investigate how this definition corresponds with least weakly absolute majorization.

\subsection{Regularized regression estimators}

Regression is an important and widely used method for establishing relationships between dependent and independent variables. In regression problems, regularization is often applied to enforce a specific desirable structure among the regression coefficients (see, e.g., \cite[Chapter~3]{Hastie2009}. Examples of this are reduce the number of relevant independent variables (structured sparsity) or to adhere to a known structure within the input data (e.g., spatio-temporal relationships between independent variables). 

The goal of this section is to demonstrate the potential of our characterization results in establishing new properties of common regression estimators such as the LASSO and ridge regression estimators. To this end, we consider the regularized least squares problem with input matrix $X \in \mathbb{R}^{m \times n}$, vector $y \in \mathbb{R}^m$ of outcomes, a vector $\beta \in \mathbb{R}^n$ of coefficients, a regularizer $\Omega$, and a regularization parameter $t$:
\begin{equation}
\min_{\beta \in \mathbb{R}^n} \ ||y - X\beta||_2^2 \quad \text{s.t. } \Omega(\beta) \leq t.
\label{prob_regression}
\end{equation}
Several special cases of Problem~(\ref{prob_regression}) and the corresponding regression methods are obtained when $\Omega$ is a $p$-norm with $p=1$ (LASSO), $p=2$ (ridge regression), and $p = \infty$ (max-norm regression). We also consider the following problem, which is closely related to the regularized regression problem for general loss functions $\Phi$:
\begin{align}
\min_{\beta \in \mathbb{R}^n} \ \Phi(X^{\top} y - \beta) \quad \text{s.t. } \Omega(\beta) \leq t
\label{prob_regression_schur}
\end{align}

One important application of regression problems is when performing design experiments. In such an experiment, the behavior of a system and the impact of independent variables is learned by choosing as input specific combinations of settings of the independent variables, rather of observed measurement data. When testing for all combinations of settings for a given subset of the independent variables, the resulting design is called orthogonal and the input matrix $X$ is orthonormal, i.e., $X^{\top} X = I$ (see, e.g., \cite{Bailey2008,Zurovac2012}). We show in Theorem~\ref{th_regression} that in this case the optimal estimator $\beta^*$ of the regularized regression problem~(\ref{prob_regression}) is also optimal for Problem~(\ref{prob_regression_schur}) for particular combinations of regularizers $\Omega$ and loss functions $\Phi$:

\begin{theorem}
Let $\beta^*$ denote the optimal estimator in Problem~(\ref{prob_regression}). If the input matrix is orthonormal, the following hold:
\begin{enumerate}
\item
For $\Omega(\beta) = ||\beta||_1$ or $\Omega(\beta) = ||\beta||_{\infty} = \max_{i \in N} |\beta_i|$, $\beta^*$ is also an optimal estimator for Problem~(\ref{prob_regression_schur}) for any choice of monotonically even continuous Schur-convex function $\Phi$;
\item
If $\Omega(\beta) = ||\beta||_2$, there exists an outcome vector $y \in \mathbb{R}^n$ and a continuous Schur-convex function $\Phi$ that is either non-increasing, non-decreasing, or monotonically even such that $\beta^*$ is not an optimal estimator for Problem~(\ref{prob_regression_schur});
\end{enumerate}
\label{th_regression}
\end{theorem}
\begin{proof}
If $X^{\top} X = I$, the objective function of Problem~(\ref{prob_regression}) reduces to
\begin{align*}
||y - X\beta||_2^2
&=
y^{\top} y - 2y^{\top} X\beta + \beta^{\top} X^{\top} X \beta
=
y^{\top} X^{\top} X y - 2y^{\top} X\beta + \beta^{\top} \beta \\
&=
|| X^{\top} y - \beta||_2^2
= \sum_{i \in N} (x_i^{\top} y - \beta_i)^2,
\end{align*}
where $x_i$ is the $i^{\text{th}}$ column of $X$. Thus, $\beta^*$ optimizes a separable quadratic function over the region defined by $\Omega(\beta)$. This allows us to prove each part of the lemma as follows:
\begin{enumerate}
\item
Each of the regions defined by the constraints $||\beta||_1 \leq t$ and $||\beta||_{\infty} \leq t$ can be reformulated as a bisubmodular polyhedron $\tilde{B}(h)$ where $h(X,Y) = t$ and $h(X,Y) = t |X \cup Y|$ for all $(X,Y) \in 3^N$, respectively. In both cases, it follows from Lemma~\ref{lemma_even_bi} that $\beta^*$ is the (unique) least weakly absolutely $(1,X^{\top} y)$-majorized element of $\tilde{B}(h)$ and the result follows from the characterization in Theorem~\ref{th_bi}.

\item
The region defined by the constraint $||\beta||_2 \leq t$ is not polyhedral and thus cannot be a base or bisubmodular polyhedron. Moreover, it cannot simultaneously be contained in a sub- or supermodular polyhedron and contain the corresponding base polyhedron. The result follows from the characterizations in Theorems~\ref{th_super}, \ref{th_sub}, and \ref{th_bi}.

\end{enumerate}
\end{proof}

The objective function of Problem~(\ref{prob_regression_schur}) closely resembles that of a regression problem with Schur-convex loss function. An interesting question, which we leave for future work, is whether Theorem~\ref{th_regression} can be extended to such problems.

\section{Conclusions and directions for future research}
\label{sec_concl}

In this article, we studied the existence of solutions to optimization that optimize whole classes of utility or cost functions simultaneously. In particular, motivated by applications in power management, we aimed to characterize the set of problems for which such solutions, called least majorized elements, exist. To answer this question, we introduced a new natural generalization of majorization that is determined by two input vectors $a$ and $b$, called $(a,b)$-majorization, and the corresponding least $(a,b)$-majorized elements. We showed that such elements exist for any valid choice of $a$ and $b$ if and only if the feasible set of the optimization problem is a base polyhedron. Similar characterizations were obtained for weaker concepts of least $(a,b)$-majorized elements and sets related to base polyhedra such as submodular, supermodular, and bisubmodular polyhedra. On the one hand, these characterizations reveal new and unique properties of these polyhedra. On the other hand, our results suggest that for most optimization problems arising in applications, no solutions exists that simultaneously optimize classes of objective functions. Although this observation is usually easy to confirm empirically for a given optimization problem, we now provide theoretical insight into why this is the case.

Given an optimization problem whose feasible set does not fall in one of the classes described above, the perhaps next best thing that one might hope for is the existence of solutions that simultaneously \emph{approximately minimize} entire classes of objective functions, instead of \emph{minimize}. Therefore, our main direction of future research will be to investigate the former topic further. More specifically, we will try to find a complete characterization of problems that do have such simultaneously approximately minimizing solutions, parametrized by the desired error factor or term of approximation. Based on existing work in this direction (e.g., \cite{Goel2006}), we expect that such a characterization will include very general classes of resource allocation problems and will thus be useful in many applications in telecommunications and energy-efficient scheduling.

We conclude this article by listing three other directions for future research:

\begin{enumerate}
\item
One limitation of this work is that we focused on the case where least $(a,b)$-majorized elements exist for \emph{each} choice of $a$ and $b$. Moreover, the definition of these elements requires them to be optimal for \emph{any} choice of separable convex objective structure. Thus, it is interesting to investigate whether the existence of least $(a,b)$-majorized elements for specific (sets of) values of $a$ and $b$ can be characterized.  In particular, an interesting question is whether such a characterization exists for $b=0$, i.e., when we only consider scaled objective functions. 

\item
Another limitation is that we focused on the case where the feasible sets are either compact and convex or bounded and hole-free, i.e., contain exactly all integral points in its convex hull. An interesting direction for future research is to investigate whether our results can be extended to feasible sets that do not have such a convexity property. Promising candidates for such sets and corresponding optimization problems include extensions of base polyhedra and RAPs, e.g., jump systems \cite{Bouchet1995} and semi-continuous knapsack polytopes \cite{deFarias2013,SchootUiterkamp2018}. Another promising class of problems consists of those whose structure and optimal solutions depend primarily on the convexity of the objective function rather than explicitly on the function itself, e.g., discrete speed scaling problems \cite{Gerards2016}.

\item
Our results indicate a natural connection between $(a,b)$-majorization and base polyhedra and between special cases of $(a,b)$-majorization and generalizations of base polyhedra. A natural follow-up question is whether more of such pairs exist. A first class of suitable candidates for sets in such a pair are polyhedra that are in some sense obtained from submodular or base polyhdra, examples of which not already discussed in this article are polybasic polyhedra \cite{Fujishige2004}, skew-bisubmodular polyhedra \cite{Fujishige2014}, and $k$-submodular polyhedra \cite{Huber2012}. This is because all our characterization results depend on the existence of least $(a,b)$-majorized elements for base polyhedra.

Another way to identify promising candidate sets are those for which a greedy algorithm in the style of Edmonds' \cite{Edmonds2003} is optimal for linear optimization. This can be seen by considering the role of linear optimization in the conditional gradient method for solving convex optimization problems with continuously differentiable objective functions (see also \cite{Bach2013,Jaggi2013}). It can be shown that, when optimizing the function $\sum_{i \in N} a_i \phi (\frac{x_i + b_i}{a_i} )$ over a base (bisubmodular) polyhedron for given vectors $a$ and $b$, there exists a sequence of iterate solutions that is a valid possible outcome of the classical Frank-Wolfe algorithm \cite{Frank1956} for any choice of continuously differentiable (even) convex function $\phi$. We expect that similar analyses can be done for polyhedra for which similar greedy algorithms are optimal for linear optimization such as (generalized) skew-bisubmodular polyhedra \cite{Fujishige2014} and polyhedra    defined    by Monge \cite{Burkard1996} and greedy \cite{Faigle2011} matrices.
\end{enumerate}

\bibliographystyle{plain}
\bibliography{library_maj}

\end{document}